\documentclass{ws-m3as}

\usepackage[super]{cite}

\usepackage{verbatim}
\usepackage{latexsym,mathtools}
\usepackage{mathtools,bm,comment,amssymb,enumitem,accents,multirow}
\usepackage{eucal}
\usepackage{cases}

\usepackage{mathtext}
\usepackage[cp1251]{inputenc}
\usepackage[T2A]{fontenc}

\usepackage{algpseudocode}
\usepackage{algorithm}
\usepackage{caption}
\usepackage{subcaption}

 \DeclareMathOperator{\mes}{mes}

\begin{document}

\markboth{M. Chugunova, H. Ji, R. Taranets and  N. Vasylyeva}{Analysis of a Radiotherapy Model for Brain Tumors}

%
\catchline{}{}{}{}{}
%

\title{ Analysis of a Radiotherapy Model for Brain Tumors}


\author{Marina Chugunova\footnote{Corresponding author}}

\address{Claremont Graduate University,
150 E. 10th Str., \\
CA 91711 Claremont, USA\\
marina.chugunova@cgu.edu}

\author{Hangjie Ji}
\address{North Carolina State University, 
3212 SAS Hall, \\
NC 27695 Raleigh, USA,\\ 
hangjie\underline{\ }ji@ncsu.edu}

\author{Roman Taranets}
\address{Institute of Applied Mathematics and Mechanics of NAS of Ukraine
G.Batyuka Str.\ 19, \\
84100 Sloviansk, Ukraine\\
taranets\underline{\ }r@yahoo.com}

\author{Nataliya Vasylyeva}
\address{Institute of Applied Mathematics and Mechanics of NAS of Ukraine
G.Batyuka Str.\ 19, \\
84100 Sloviansk, Ukraine; and\\
Dipartimento di Matematica, Politecnico di Milano,
Piazza Leonardo da Vinci 32, \\
20133 Milano, Italy\\
nataliy\underline{\ }v@yahoo.com}



\maketitle

\begin{history}
\received{(Day Month Year)}
\revised{(Day Month Year)}
\comby{(xxxxxxxxxx)}
\end{history}



\begin{abstract}
In this work, we focus on the analytical and numerical study of a
mathematical model for brain tumors with radiotherapy influence.
Under certain assumptions on the given data in the model, we prove
existence  and uniqueness of a  weak nonnegative (biological
relevant) solution. Then, assuming only more regular initial data,
we obtain the extra regularity of this solution.  Besides, we
analyze the optimal control of the advection coefficient responding
for the radiotherapy effect on the tumor cell population. Finally,
we provide numerical illustration to all obtained analytical
results.
\end{abstract}



\section{ Introduction }\label{s1}

\noindent Mathematical modeling of optimal control for radiotherapy
of cancer involves the development and analysis of various
models to improve the precision and effectiveness of treatment.
These models are essential for estimating the optimal radiation dose
and scheduling, thereby maximizing tumor control while minimizing
damage to healthy tissues.

The choice of mathematical models significantly impacts the
estimation of the minimum radiation dose required for effective
tumor control. Kutuva et al. emphasized the importance of a model
selection in radiotherapy planning, highlighting how different
models can influence dose estimations \cite{kutva2023}.

Stochastic processes provide a robust framework for modeling
radiation - induced cell survival and tumor response. Hanin
introduced an iterated birth and death process to describe the
stochastic nature of cell survival after radiation exposure,
offering a foundational model in this area \cite{hanin2001}.
Further, Hanin  developed a stochastic model to analyze tumor
response to fractionated radiation therapy, focusing on limit
theorems and convergence rates \cite{hanin2004}.

Alternative approaches to the traditional linear-quadratic model
have been proposed to better understand radiation effects.
Hanin$\&$Zaider  presented a mechanistic description of cell
survival probability at high radiation doses, providing an
alternative framework that improves upon the linear-quadratic model
\cite{hanin2010}. In another study, Hanin$\&$Zaider described the
kinetics of radiation-induced damage and healing in normal tissues,
contributing to the development of more accurate mechanistic models
for radiotherapy \cite{hanin2013}.

Optimization strategies for fractionated radiation therapy have been
explored through various ma\-the\-ma\-ti\-cal principles.
Hanin$\&$Zaider
 proposed optimal radiation schedules using the greedy
principle, demonstrating a biologically-based adaptive boosting
approach to improve treatment efficacy \cite{hanin2014}. Hanin,
Rachev$\&$Yakovlev addressed the optimal control of a cancer
radiotherapy for non-homogeneous cell populations, employing
advanced probabilistic methods to tailor treatment strategies
\cite{hanin1993}.

Comprehensive stochastic models integrating multiple biological
factors have also been developed. Hanin et al.  presented a
comprehensive model for irradiated cell populations in culture,
accounting for various stochastic and biological processes
\cite{hanin2006}. Hlatky, Hahnfeldt$\&$Sachs  examined the influence
of time-dependent stochastic heterogeneity on cell population
radiation response, highlighting the complexity of tumor dynamics
\cite{hlatky1994}.

Incorporating cell cycle effects into tumor control probability
models has improved predictions of treatment outcomes. Hillen et al.
 linked cell population models with tumor control probability,
considering the impact of cell cycle dynamics on radiotherapy
efficacy \cite{hillen2010}. Iwasa, Nowak$\&$Michor  discussed the
evolution of resistance during clonal expansion, providing insights
into the adaptive nature of tumors under radiation therapy
\cite{iwasa2006}.

Recent studies have explored novel mechanisms influencing tumor
behavior and treatment response. Gerlee et al.  explained how
autocrine signaling can lead to Allee effects in cancer cell
populations, offering a new perspective on tumor growth dynamics
\cite{gerlee2022}. Clinical studies, such as Byun et al., evaluated
the prognostic potential of mid-treatment nodal response in
oropharyngeal squamous cell carcinoma, contributing to personalized
treatment approaches \cite{byun2020}.

Gene expression models have also been employed to predict tumor
radiosensitivity and treatment prognosis. Eschrich et al. developed
a gene expression model to forecast response and prognosis after
chemoradiation, highlighting the role of molecular markers in
treatment planning \cite{eschrich2009}. Hendry$\&$Moore investigated
the factors contributing to the steepness of dose-incidence curves
for tumor control, examining whether variations arise before or
after irradiation \cite{hendry1984}.

Additional studies have focused on the effects of combined therapies
and spatial optimization. Powathil et al.  modeled the effects of
radiotherapy and chemotherapy on brain tumors, providing insights
into combined treatment strategies \cite{powathil2007}. Billy,
Clairambault$\&$Fercoq  optimized cancer drug treatments using cell
population dynamics, emphasizing the role of mathematical methods in
drug scheduling \cite{billy2013}. Bratus et al.  maximized viability
time in a mathematical model of cancer therapy, demonstrating
optimization techniques for prolonging treatment effectiveness
\cite{bratus2017}. Meaney et al.  addressed spatial optimization for
radiation therapy of brain tumors, contributing to the precision of
treatment delivery \cite{meaney2019}. Swanson et al.  used
mathematical modeling to quantify glioma growth and invasion,
offering a virtual platform to simulate tumor dynamics and optimize
therapy \cite{swanson2003}.

The article by Rockne et al.,  models the effect of radiotherapy on
glioma patients, demonstrating how mathematical modeling can predict
tumor response and assist in treatment planning
\cite{rockne2008modeling}. In another study, Rockne et al. apply a
mathematical model to predict the efficacy of radiotherapy in
individual glioblastoma patients, providing a personalized approach
to treatment and highlighting the potential for improved clinical
outcomes through tailored therapy \cite{rockne2010predicting}.

In this paper, motivated by the discussion above, we focus on the
analytical and numerical study of  the model suggested by Rockne et
al. \cite{rockne2010predicting}. Let $\Omega\subset\mathbb{R}^{d},$
with $d=1,2,3,$ be a bounded domain with a smooth boundary
$\partial\Omega$, and let positive $T$ be an arbitrarily fixed final
time. Setting
\[
\Omega_{T}=\Omega\times(0,T)\qquad\text{and}\qquad
\partial\Omega_{T}=\partial\Omega\times(0,T],
\]
we consider the following nonlinear parabolic equation with the
unknown function (describing tumor cell density)
$u=u(x,t):\Omega_{T}\mapsto\mathbb{R},$
\begin{equation}\label{c-1}
u_{t} - \nabla (D(x) \nabla u ) = a(x,t) u(1-u)  \qquad\text{ in
}\quad \Omega_T,
\end{equation}
coupled with the homogenous  Neumann boundary condition
\begin{equation}\label{c-2}
 \nabla u \cdot \textbf{n}  = 0 \qquad\quad  \text{ on }\quad \partial \Omega_T,
\end{equation}
and the initial data
\begin{equation}\label{c-3}
 u(x,0) = u_0(x) \qquad\text{in }\quad \Omega.
\end{equation}
Here $\textbf{n}$ is the unit outward normal vector to
$\partial\Omega;$ $D(x)$ is the net invasion rate; $a(x,t) = \rho -
R(x,t),$ where $\rho$ is the net proliferation rate, while $R(x,t)$
is the loss term represents the effects of the external beam
radiation therapy; the function $u_0(x)$ is the initial tumor cell
density. The no-flux boundary condition (\ref{c-2}) is imposed to
prevent tumor cells from leaving the brain domain at its boundary
$\partial \Omega$.

We assume that tumor cells only spread in brain tissue and do not
diffuse into any other tissue that is not segmented as gray or white
matter. In order to model the reduced gray matter infiltration in
comparison to white matter, the invasion rate $D$ is dependent on
$x$. In \cite{swanson2003}, Swanson et al. took  into account the
brain tissue heterogeneity
$$
D(x)=\left\{\begin{array}{rc}
D_g & \text {in grey brain tissue, } \\
D_w & \,\text { in white brain tissue, }
\end{array}\right.
$$
where $D_w>>D_g>0$ corresponding to the observation that tumor cells
diffuse faster on white matter.

Note that, the difference between the net proliferation rate and the
loss term (i.e. $a(x,t)$) reflects the overall impact of external
beam radiation therapy on cell growth and survival. A positive
difference indicates that cell proliferation is outpacing cell loss,
which may suggest resistance to radiation therapy. Conversely, a
negative difference suggests that cell loss is greater than
proliferation, indicating a more effective response to treatment.

To determine the loss term $R(x,t)$, the classic linear-quadratic
model is used to quantify radiation efficacy (see
\cite{rockne2008modeling, rockne2010predicting}):
$$
E=\mu \text { Dose }+\nu\text { Dose }^2,
$$
in which the radiation dose (defined in both space and time) is
related to an effective dose. The coefficients $\mu$ and $\nu$ are
the radiobiology parameters and determine the relative contribution
of each term in the sum toward the total radiation effect and the
ratio of the parameters $\mu / \nu$ represents the tissue response.
The ratio $\mu / \nu$ as for many tumors is considered to be in the
order of 10 Gy for brain tumors \cite{rockne2008modeling}. We set
$\mu / \nu=10$ Gy and so $\mu$ is the sole radiation efficacy
parameter. As in \cite{rockne2008modeling, rockne2010predicting}, we
assume that the effect of concurrent radiotherapy is included in the
net response parameter $\mu$.

The probability of survival of glioma cells after the administration
of radiation is defined in the following way
$$
\mathcal{S}=\exp (-E),
$$
where the larger the dose, the smaller the probability of survival.
Clinically, normal radiotherapy is administered in a series of
fractions in order to maximize tumor response while minimizing
normal tissue toxicity. This frequently results in daily dosages
ranging from $1$ to $2.5 \mathrm{~Gy}$. Consequently, $R(\mathrm{x},
t)$ is defined as
$$
R(\mathrm{x}, t)= \begin{cases}0 & \text { for } t \notin \text {
therapy, } \\ 1-\mathcal{S}(\mu, \operatorname{Dose}(\mathrm{x},
\mathrm{t})) & \text { for } t \in \text { therapy, }\end{cases}
$$
representing the effect of radiotherapy on the tumor cell population
which is given by the probability of death, one minus the
probability of survival.

The partial differential equation (\ref{c-1}) with $D(x)$ and $a(x,
t)$ being positive constants is the well-known
Kolmogorov-Piskunov-Petrovsky (KPP) equation. The KPP equation, also
called as the Fisher-KPP equation, has been extensively exploited
and studied in  mathematical biology and nonlinear dynamics (see,
e.g.  \cite{KPP37,Fis37}). Investigation of initial-boundary value
problems (IBVPs) associated with  equations similar to \eqref{c-1}
including well-posedness, stability of corresponding solutions, the
existence of traveling wave solutions, and various applications of
the KPP in population dynamics, epidemiology and pattern formation
have been discussed in \cite{Grin91,Mur02,Mur03} (see also
references therein). Some key results include the description of
conditions providing solvability, uniqueness, the characterization
of wave-front propagation speed and the analysis of bifurcations and
stability of steady states (see, e.g.
\cite{kpp1,kpp2,BF18,HR16,FM77}). The most of the outcomes reported
in the aforementioned papers concern with the models having constant
or location-dependent coefficients. To the best author's knowledge,
equations similar \eqref{c-1} with $D$ and $a$ being functions are
not  actively analytical studied. We refer the very recent article
\cite{WXZZ21} in which the authors utilize the initial-boundary
value problem to \eqref{c-1} subject to the homogeneous Neumann
boundary condition to describe the diffusive SIS model in
mathematical epidemiology. In particular,  under more regular
strictly positive initial data ($u_0 \in W^{2,2}(\Omega)$) and a
bounded positive $a(x,t)$, the existence of weak solutions to this
problem and the stabilization of these solutions are established.
Finally, we mention optimal control problems for the Fisher-KPP type
equations which are actively discussed for different type of
``cost'' functional (see, e.g.
\cite{FLV12,CGM,DFLLY10,rockne2010predicting,kao2}). In particular,
the work  \cite{rockne2010predicting} analyzes  an optimal control
problem to (\ref{c-1}) with different optimization functionals  and
with a number of additional conditions like, for example, short- and
long-time existence of nonnegative weak solutions (without full
details).

Having said that the picture of investigations is now pretty clear,
there still remain some important unexplored issues related with the
 analytical study of \eqref{c-1}--\eqref{c-3} in the case
of the time and space depending coefficients having weaker
regularity. The present paper aims to fill this gap in the
 qualitative  analysis of \eqref{c-1}--\eqref{c-3}. The
main achievements of this art can be summarized in the following
three points.

\noindent\textbf{I:}\textit{ Well-posedness and regularity of a
solution to \eqref{c-1}--\eqref{c-3}} Assuming a weaker regularity
of given functions in \eqref{c-1}--\eqref{c-3} and making the
non-sign condition on the coefficients $a(x,t),$ we discuss the
 local and the global
one-to-one weak solvability of the problem. Moreover, we describe
assumptions on the initial data $u_{0}$ which allow us to get a
nonnegative bounded weak solution in \eqref{c-1}-\eqref{c-3}
possessing extra regularity. Utilizing the obtained results and
requiring ceratin additional assumptions on $u_{0}$, we prove the
existence of the absorbing set in the corresponding functional
classes. In biological view point, the latter means that the
biologically relevant domain $\mathfrak{R}=\{u\in L^{2}(\Omega):\,
0\leq u\leq 1\}$ is invariant for evolution problem
\eqref{c-1}--\eqref{c-3}. The main tools exploited in this
investigation are Galerkin approximation scheme, energy and entropy
estimates, nonlinear Gr\"{o}nwall inequality and Stampacchia approach
modified to our targets. In fine, we notice that the most of these
qualitative outcomes ply a key role to reach our second goal in this
study.

\smallskip
\noindent\textbf{II:}\textit{ Optimal control in
\eqref{c-1}--\eqref{c-3}.} We discuss the problem of searching the
control $u$ which provides the optimal effect of the radiation
therapy to treat a  tumor in the brain describing via mathematical
model \eqref{c-1}-\eqref{c-3}. From the mathematical point of view,
this problem consists in minimizing a certain cost functional $J(R)$
subject to \eqref{c-1}--\eqref{c-3} in a given class of admissible
controls, $R\in\mathcal{U}$. Here, we prove the existence of the
optimal control $R^{*}$ and describe its main properties including
the sensitivity, the necessary condition and the uniqueness of
$R^{*}$ (depending on the certain assumptions on the constrain). On
this route, we incorporate  standard techniques (with some
modifications) of the control theory to parabolic equations.

\smallskip
\noindent\textbf{III:}\textit{ Quantitative analysis in
\eqref{c-1}--\eqref{c-3}.} We complement our analytical results by numerically simulating the optimal control problem discussed in II. We formulate the optimal control problem as a PDE-constrained optimization problem with additional constraints imposed on the admissible controls $R\in\mathcal{U}$. To identify the optimal control $R^*$, we perform gradient descent on the optimization problem and verify the algorithm using both one-dimensional and two-dimensional numerical examples.


\subsection*{Outline of the paper.} In Section \ref{s2}, we
introduce the notation and the functional setting along with some
key inequalities. In Section \ref{s3}, we state the main results
concerning with the local and global weak solvability of
\eqref{c-1}--\eqref{c-3} (Theorems \ref{Th-ex} and \ref{t3}) and the
regularity of this solution (Theorems \ref{Th-ex1} and \ref{Th-2},
Corollary \ref{R-1}). The verification of these results are carried
out in Sections \ref{s4}-\ref{s7}. In particular, the local weak
solvability, the uniqueness and the nonnegativity of $u$ stated in
Theorem \ref{Th-ex} are obtained in Section \ref{s4}.  Section
\ref{s5} is devoted to the proof of Theorem \ref{Th-ex1} related
with the boundedness of a non-negative weak solution. The higher
regularity of $u$ established in Theorem \ref{Th-2} are examined in
Section \ref{s6}. The verification of one-to-one global weak
solvability along with the extra regularity of $u$ (Theorem
\ref{t3}) is executed in Section \ref{s7}. Section \ref{s8} is
dedicated to the optimal control study: the existence of $R^{*}$
(Theorem \ref{lem-e}); the sensitivity of the map $R\mapsto u$
(Lemma \ref{lem-s}); the necessary condition of the optimality
(Theorem \ref{Th-opt}) and the uniqueness of $R^{*}$ (Theorem
\ref{Th-un}). In the final Section \ref{s9}, we treat \eqref{c-1}--\eqref{c-3} from the numerical side.


\section{Functional Setting and Notations}\label{s2}

Throughout this paper, for any $p\geq 1,$ $s\geq 0,$ and  any Banach
space $(\mathbf{X},\|\cdot\|_{\mathbf{X}}),$ we exploit the usual
spaces
\[
W^{s,p}(\Omega),\quad L^{p}(\Omega),\quad
L^{p}(0,T;\mathbf{X}),\quad W^{s,p}(0,T;\mathbf{X}),\quad
C([0,T];\mathbf{X}).
\]
For notational simplicity, we denote
\[
\|\cdot\|=\|\cdot\|_{L^{2}(\Omega)}\quad \text{and}\quad
H^{1}(\Omega)=W^{1,2}(\Omega),
\]
besides the brackets $(\cdot,\cdot)$ and $\langle\cdot,\cdot\rangle$
stand for the scalar product in $L^{2}(\Omega)$ and the duality
pairing on $(H^{1}(\Omega))^{*}\times H^{1}(\Omega)$,
correspondingly.

\noindent Let $\mathbf{H}$ be a real separable Hilbert space with a
norm $\|\cdot\|_{\mathbf{H}}$ and $\mathbf{Y}$ be a Hilbert space
such that
\[
\mathbf{Y}\hookrightarrow \hookrightarrow\mathbf{H}\hookrightarrow
\mathbf{Y}_{1}
\]
is a Gelfand triple.

In the existence proof, the parabolic embedding is typically
employed. In particulary, for $q,p$ and $r$ satisfying the relations
\[
\frac{1}{r}+\frac{d}{pq}=\frac{d}{p^{2}},\qquad\
\begin{cases}
q\in\big[p,\frac{dp}{d-p}\big],\quad
r\in[p,+\infty]\quad\quad\text{if}\quad 1\leq p<d,\\
q\in[p,+\infty),\quad r\in\big(\frac{p^{2}}{d},+\infty\big]\quad
\text{if}\quad 1<d\leq p,
\end{cases}
\]
\cite[Proposition 3.4, Chapter 1]{DiB} establishes the following
compact embedding
\begin{equation}\label{2.0}
\|v\|_{L^{r}(0,T;L^{q}(\Omega))}\leq
c_{0}(1+T|\Omega|^{-p/d})^{1/r}\|v\|_{V^{p}(\Omega_{T})},
\end{equation}
where $c_{0}$ is a positive constant depending only on $d,p$ and the
structure $\partial\Omega$, and
\[
V^{p}(\Omega_{T})=\{v\in L^{\infty}(0,T;L^{p}(\Omega))\cap
L^{p}(0,T; W^{1,p}(\Omega))\}.
\]
 Moreover, \cite[Proposition 3.2, Chapter 1]{DiB} provides the estimate
 \begin{equation}\label{2.0*}
\|v\|_{L^{q}(\Omega_{T})}\leq
c_{0}(1+T|\Omega|^{-\frac{p}{d}})^{1/q}\|v\|_{V^{p}(\Omega_{T})}
 \end{equation}
for  $q=\frac{p(d+p)}{d}. $

We also utilize the following continuous embedding in the
time-continuous functional space to establish additional regularity
of a solution to the initial-boundary value problem,
\[
\mathcal{W}_{2,2}^{1}(0,T; \mathbf{Y},\mathbf{Y}_{1})=\bigg\{v\in
L^{2}(0,T;\mathbf{Y}):\quad \frac{\partial v}{\partial t}\in
L^{2}(0,T;\mathbf{Y}_{1})\bigg\}\hookrightarrow
C([0,T];[\mathbf{Y},\mathbf{Y}_{1}]_{1/2}),
\]
where $[\mathbf{Y},\mathbf{Y}_{1}]_{1/2}$ denotes the interpolation
space between $\mathbf{Y}$ and $\mathbf{Y}_{1}$ of order $1/2$  (see
\cite[Theorem 3.1, Chapter 1]{LM}), e.g.
\[
[H^{1}(\Omega),(H^{1}(\Omega))^{*}]_{1/2}=L^{2}(\Omega),
\]
which in turn means
\begin{equation}\label{i.1}
\mathcal{W}_{2,2}^{1}(0,T;
H^{1}(\Omega),(H^{1}(\Omega))^{*})\hookrightarrow
C([0,T];L^{2}(\Omega)).
\end{equation}

Finally, we conclude this section with some inequalities that will
play a key role in the getting well-posedness and regularity of a
solution to \eqref{c-1}-\eqref{c-3}:

\noindent$\diamondsuit$ \textbf{Gagliardo-Nirenberg inequality:} for
every $v=v(x)\in L^{p}(\Omega)$, $p> 1,$
\[
\|\mathcal{D}_{x}^{j}v\|_{L^{p}}\leq
c_{1}\|\mathcal{D}_{x}^{m}v\|^{\theta}_{L^{r}}\|v\|^{1-\theta}_{L^{q}}+c_{2}\|v\|_{L^{q}}
\]
for $1\leq q\leq\infty $, $r> 1$, $0\leq j<m$ and $j/m\leq
\theta\leq 1$ satisfying
\[
\frac{1}{p}=\frac{j}{d}+\theta\Big(\frac{1}{r}-\frac{m}{d}\Big)+\frac{1}{q}(1-\theta).
\]
In particular case, this inequality reads as
\begin{equation}\label{g-n-0}
\| v  \|_{L^p (\Omega )}  \leqslant c_1  \| \mathcal{D}_{x} v
\|^{\theta }  \| v  \|^{1 - \theta } + c_2 \| v \|
\quad\text{with}\quad \theta = \frac{d(p -2)} {2p} \in (0,1]  ,
\end{equation}
where $p > 2$ and $c_2 = 0$ if $\Omega$ is unbounded or $v$ has a
compact support.

\noindent$\diamondsuit$ \textbf{Jensen  inequality for sum:} for
each non-negative number $b_{i},$ $i=1,2,...,M,$ and positive $p$,
\[
A_{1}\sum_{j=1}^{M}b^{p}_{j}\leq
\Big(\sum_{j=1}^{M}b_{j}\Big)^{p}\leq A_{2}\sum_{j=1}^{M}b_{j}^{p}
\]
with $A_{1}=\min\{1,M^{p-1}\}$ and $A_{2}=\max\{1,M^{p-1}\}$.

\noindent$\diamondsuit$ \textbf{Nonlinear generalization of
Gr\"{o}nwal inequality:} if a non-negative function $v=v(t)$
satisfies the  inequality
\[
v(t)\leq c_{3}+\int_{t_{0}}^{t}[b_{0}(s)v(s)+ b(s)v^{\gamma}(s)]ds
\]
with non-negative continuous functions $b_{0}(s),$ $b(s)$ for
$t>t_{0}$, $\gamma>1$ and
\[
0\leq
c_{3}<\bigg[(\gamma-1)\int_{t_{0}}^{t_{0}+h}b(s)ds\bigg]^{-\frac{1}{\gamma-1}}\exp\bigg\{-\int_{t_{0}}^{t_{0}+h}b_{0}(s)ds\bigg\},\quad
h>0,
\]
then
\begin{equation}\label{2.3}
v(t)\leq
c_{3}\bigg[\exp\bigg\{(1-\gamma)\int_{t_{0}}^{t_{0}+h}b_{0}(s)ds\bigg\}
-c_{3}^{-1}(\gamma-1)\int_{t_{0}}^{t}b(s)ds
\bigg]^{\frac{1}{\gamma-1}}
\end{equation}
for $t\in[t_{0},t_{0}+h].$

This inequality is particular case of  \cite[Theorem 21]{Dr}
rewritten here to our target.

\section{Well-posedness and Regularity: Statements of the Main Results}\label{s3}

\noindent For readers' convenience, we begin this section with
recalling the pretty standard definition of a weak solution to
initial-boundary value problem \eqref{c-1}--\eqref{c-3}.

\begin{definition}\label{weak-sol}
A function $u=u(x,t):\Omega_{T}\mapsto\mathbb{R}$ is  called a
strong solution to problem \eqref{c-1}--\eqref{c-3} if

\noindent$\bullet$ $u\in V^{2}(\Omega_{T})$ and $\frac{\partial
u}{\partial t}\in L^{2}(0,T;(H^{1}(\Omega))^{*})$;

\noindent$\bullet$ the initial condition \eqref{c-3} holds in
$L^{2}(\Omega)$;

\noindent$\bullet$ for any positive $T$ and each function $\phi\in
L^{2}(0,T;H^{1}(\Omega)),$ the identity holds
\begin{equation}\label{ident}
\int \limits_{0}^T \langle \partial_t u, \phi \rangle \, dt  + \int
\limits_{\Omega_T}  D (x)  \nabla u   \nabla \phi \,dx dt= \int
\limits_{\Omega_T}  a(x,t)  u  (1 - u ) \phi \,dx dt.
\end{equation}
\end{definition}

\noindent Throughout this paper, we will assume that the coefficient
$D=D(x)$ meets the following requirement
\begin{equation}\label{2.2}
D\in L^{\infty}(\Omega)\qquad\text{and}\qquad D\geq d_{0}>0
\end{equation}
with some positive constant $d_{0}$.

Our first result related with the local (in time) solvability of
problem \eqref{c-1}--\eqref{c-3}.

\begin{theorem}\label{Th-ex}\verb"(Local Solvability)"
Let $d=1,2,3,$  a nonnegative $ u_0(x) \in L^{2}(\Omega)$ and
$a(x,t) \in L^{\infty}(\Omega_{\infty})$. Then under assumption
\eqref{2.2}, problem \eqref{c-1}-\eqref{c-3} admits a unique
non-negative weak local solution $u=u(x,t)$ defined in
$\Omega_{T^{*}}$, where positive $T^{*}$ depends only on
 $\|u_0\|$, $\| a \|_{L^{\infty}(\Omega_{\infty})}$; the Lebesgue
 measure of
$\Omega$; the parameters  $d$ and $d_{0}$.
\end{theorem}

In the next assertion, selecting only more regular initial data, we
establish the local (in time) boundedness of a weak solution.
\begin{theorem}\label{Th-ex1}\verb"(Local Regularity)"
Let $d, D(x)$ and $a(x,t)$ meet the requirements of Theorem
\ref{Th-ex}, and the non-negative $ u_0(x) \in L^{\infty}(\Omega)$.
Then there are a positive time $T^{**}\in(0,T^{*}]$ and a positive
constant $C_{0}$ depending only on the corresponding norms of
 $ u_0(x), a(x,t);$  the quantities $d, d_{0}$ and the Lebesgue
 measure of
$\Omega$, such that a  nonnegative weak solution $u(x,t)$ of
\eqref{c-1}--\eqref{c-3} satisfies the bound
$$
\| u \|_{L^{\infty}(\Omega_{T^{**}})} \leqslant C_0.
$$
\end{theorem}
The following results are related with sharper local estimates of a
weak solution whenever some additional assumptions on the initial
data are done.

\begin{theorem}\label{Th-2}\verb"(Higher Regularity)"
Let the following estimates hold
\begin{equation}\label{i.2}
0 \leqslant u_0(x) \leqslant 1  \qquad\text{and} \quad   \int
\limits_{\Omega}{(|\ln(u_0)| + | \ln(1-u_0)| ) \,dx} < \infty.
\end{equation}
Then, under assumptions of Theorem \ref{Th-ex}, a nonnegative weak
solution constructed in Theorem \ref{Th-ex} satisfies the estimates
\begin{align}\label{entr-e}\notag
&0 \leqslant u(x,t) \leqslant 1 \qquad\text{ a.\,e. in } \Omega_{T^*},\\
&\int \limits_{\Omega}{( |\ln u(x,t) | + | \ln(1-u(x,t) )| ) \,dx} <
\infty
\end{align}
for all $t\in [0,T^*]$. Besides, $\ln u$ and $ \ln(1-u)$ belong to
$L^2 (0, T^*; H^1(\Omega))$.
\end{theorem}

As for the global (in time) weak solvability of
\eqref{c-1}--\eqref{c-3} (that is a weak solvability for any fixed
positive $T$), this occurs if
 the coefficient $a(x,t)$ is nonnegative or $u_{0}(x)$ satisfies assumptions of Theorem \ref{Th-2}.

The first claim which is a simple consequence Theorems \ref{Th-ex}
and \ref{Th-2} concerns the global solvability if $u_{0}$ has
additional regularity.
\begin{corollary}\label{c-4}
Let $a(x,t),$ $D(x)$ and $d$ meet requirements of Theorem
\ref{Th-ex}. Then, for any positive $T$, under assumptions
\eqref{i.2}, problem \eqref{c-1}--\eqref{c-3} admits a unique
nonnegative global weak solution satisfying \eqref{entr-e}.
\end{corollary}
The verification of  this statement is pretty standard and repeats
the mains steps of the proof of  Theorem \ref{Th-ex} with certain
modifications provided by Theorem \ref{Th-2}. Namely, exploiting
Theorem \ref{Th-2} allows one to refine the energy estimates stated
in Theorem \ref{Th-ex} (see \eqref{2.4**} in Section \ref{s4.2}),
which in tune removes the restriction on the choice of $T^{*}$ (see
\eqref{2.4}, Section \ref{s4.2}). After that, recasting the standard
extension arguments ends up with the global solvability established
in this corollary.

Next, outcome deals with the global solvability in the case of the
nonnegative coefficient $a(x,t)$.

\begin{theorem}\label{t3}\verb"(Global Solvability and Regularity)"
Let $d,$ $D(x)$ and a nonnegative $a(x,t)$ meet the requirements of
Theorem \ref{Th-ex}. Then the results of Theorems \ref{Th-ex} and
\ref{Th-ex1} hold for any positive fixed $T$, if only
 a nonnegative  $u_{0}\in L^{2}(\Omega)$ and $u_{0}\in L^{\infty}(\Omega)$, respectively.
\end{theorem}
\begin{remark}
The proof of this claim (see Section \ref{s7}) tells that the
assumption on the regularity of $a(x,t)$ in Theorem \ref{t3} can be
relaxed. Namely,  the requirement of a nonnegative $a\in
L^{\infty}(\Omega_{T})$ for any fixed positive $T$ provides the
fulfillment of this theorem.
\end{remark}

\begin{remark}\label{r0}
Theorem \ref{t3} (more precisely, the first inequalities in
\eqref{entr-e} and \eqref{i.2}, respectively) suggests that the
biologically relevant domain
\[
\mathfrak{R}=\{u\in L^{2}(\Omega):\, 0\leq u(x,t)\leq 1\}
\]
is invariant for evolution system \eqref{c-1}-\eqref{c-3}. Indeed,
if we take any biologically meaningful initial data
$u_{0}\in\mathfrak{R}$, then each weak solution of
\eqref{c-1}-\eqref{c-3} starting from this $u_{0}$ remains in
$\mathfrak{R}$ for each $t\in[0,T]$ with every fixed positive time
$T.$
\end{remark}

In fine, Theorems~\ref{Th-2} and \ref{t3} provide the following very
useful property of a nonnegative weak solution
 which plays a key role to study the uniqueness of the optimal control in Section \ref{s8}.

\begin{corollary}\label{R-1}
Let $a(x,t)$ and $D(x)$ meet the requirements of Theorem
\ref{Th-ex}, and let \eqref{i.2} hold. Then
$$\mes \{(x,t)\in\Omega_{T}:\, u = 0\}
= 0\quad\text{and}\quad\mes \{(x,t)\in\Omega_{T}:\, u = 1\}  =0,
$$
where  $T>0$ is any but fixed in the case of a nonnegative
$a(x,t),$ otherwise  $0<T\leq T^{*}$.
\end{corollary}

The proof of these results are carried out in Sections
\ref{s4}-\ref{s7}.

\section{The Proof of Theorem~\ref{Th-ex}}\label{s4}

\noindent In order to prove this claim, we exploit the strategy
consisting in the $4^{\text{th}}$ main steps. In the first, we
construct approximate solution $u^{N}=u^{N}(x,t)$ via the Galerkin
approximation scheme which gives local in time solvability to the
Galerkin approximating equation corresponding to \eqref{c-1}. In the
second step, we obtain uniform estimates for $u^{N}$ on $[0,T^{*}]$,
allowing us to pass to the limit and to end up with the local weak
solvability of \eqref{c-1}-\eqref{c-3}. The third step is dedicated
to prove the nonnegativity of the constructed solution for
$t\in[0,T^{*}]$. The main tools on this way are the special test
function in \eqref{ident} and nonlinear Gr\"{o}nwall inequality
\eqref{2.3}. Finally, via the standard approach, we deduce the
uniqueness of the local weak solution built via the first two
stages.


\subsection{The Faedo-Galerkin approximation scheme}\label{s4.1}

Let $\lambda_k$ and $\varphi_k,$  $k=0,1,2,\ldots$, be  the sequence
of the eigenvalues and of the normalized eigenfunctions of the
operator $-\nabla (D(x)\nabla)$ in $\Omega$  associated with the
homogeneous Neumann boundary conditions, that is
 for every $k=0,1,2,\ldots,$ the pair $(\lambda_{k},\varphi_{k})$ is the unique
solution of the problem
\begin{equation*}\label{eq:EVN}
\begin{cases}
- \nabla ( D(x) \nabla \varphi_k ) = \lambda_k \varphi_k  & \text{ in }   \Omega ,\\
\nabla \varphi_k \cdot \textbf{n} = 0
 &  \text{ on } \partial\Omega,\\
 \|\varphi_k\|=1.
\end{cases}
\end{equation*}
We recall that  $\varphi_0=1/\sqrt{|\Omega|}$ and
\begin{equation*}\label{EVlim}
    \lambda_0=0<\lambda_1\leqslant \lambda_2\leqslant \ldots, \qquad \lim_{n\to+\infty}\lambda_n=+\infty.
\end{equation*}
As already observed, the system of the eigenfunctions $\{\varphi_i
\}_{i \in \mathbb{N}}$ is an orthonormal basis in $L^2( \Omega)$ and
an orthogonal system in $H^1(\Omega)$.

Now, we are ready to define  the approximate problems to
\eqref{c-1}--\eqref{c-3}. For $N\geqslant 0$ we set the
$(N+1)$-dimensional space
\begin{equation*}\label{eq:VN}
\Pi_N  =
\text{span}\left\{\varphi_0,\varphi_1,...,\varphi_{N}\right\}.
\end{equation*}
It is apparent that  $\Pi_N\subset H^{1}(\Omega)$,  $\Pi_M\subset
\Pi_N$ for $N > M$ and besides, $\bigcup_{N=0}^{+\infty}\Pi_N$ is
dense in $H^1(\Omega)$.

Denoting the projection of the initial datum $u_{0}$ on $\Pi_N$ by
\begin{equation*}\label{eq:proju0}
    u_{0}^N(x) =  \sum_{i=0}^{N}(u_{0},\varphi_i)\varphi_i(x),
\end{equation*}
 we look for unknown function $ u^N=u^{N}(x,t) \in W^{1,2}(0,T;\Pi_N) $ solving  the system
\begin{equation}\label{eq:discretized}
\begin{cases}
\int \limits_{\Omega}  \partial_t u^N v\, dx  + \int \limits_{\Omega} { D (x)  \nabla u^N  \nabla v \,dx }=
\int \limits_{\Omega} { a(x,t)  u^N (1 - u^N) v \,dx} ,\\
 u^{N}(\cdot,0) = u^N_{0}(\cdot)  \text{ in } \Omega
\end{cases}
\end{equation}
for any $t\in (0,T)$ and for any $v \in \Pi_N$.

Substituting $v=\varphi_i,$  $i=0,...,N,$ in (\ref{eq:discretized})
and performing standard technical calculations reduce problem
(\ref{eq:discretized}) to a  system of nonlinear ordinary
differential equations for the ''components'' of the vector  $u^N$
in $\Pi_N$ with respect to the chosen basis. More precisely, we
search functions $\psi_{i}=\psi_i(t)\in W^{1,2}(0,T)$  for
$i=0,1,..., N,$ such that the function
$$
u^{N}(x,t) = \sum \limits_{i =0}^{N} {\psi_i(t) \varphi_i(x)}
$$
solves (\ref{eq:discretized}).

It is apparent that,  problem (\ref{eq:discretized}) is equivalent
to the following Cauchy problem for a system of ordinary
differential equations for $\psi_{j},$ $j=0,1,...,N,$
\begin{equation}\label{ap-01}
\begin{cases}
 \frac{ d}{ d t} \psi_j(t)  =F_{j},
 \qquad  j = 0,1,..., N,\\
\psi_j(0)  = (u_0, \varphi_j),
\end{cases}
\end{equation}
where we  set
\begin{align*}
F_{j}=F_{j}(t,\psi_{0},\psi_{1},...,\psi_{N})&= - \lambda_j
\psi_j(t) + \sum \limits_{k
=0}^{N}\psi_k(t)\int\limits_{\Omega}a(x,t) \varphi_k(x)\varphi_j(x)
\, dx
\\
& -   \int\limits_{\Omega}  a(x,t) \Big( \sum \limits_{i =0}^{N}
{\psi_i(t) \varphi_i(x)} \Big)^2   \varphi_j(x) \, dx.
\end{align*}
Performing  straightforward calculations and bearing in mind  the
regularity of $a(x,t)$, we immediately conclude that
$F=\{F_{0},...,F_{N}\}$ is locally Lipschitz. Moreover, there exists
positive constant $L$ such that the inequality
$$|F| \leqslant L (|\psi| +|\psi|^2)$$
holds for any $\psi=\{\psi_{0},...,\psi_{N}\} \in \mathbb{R}^{N+1}$.

In summary, keeping in mind the properties of $F$ and appealing  to
the Picard-Lindel\"{o}f Theorem, we immediately end up with the
unique classical solvability of the system \eqref{ap-01}  on a time
interval $[0,T_{N}]$ and $\psi_{j}\in C^{1}([0,T_{N}])$ for
$j=0,...,N.$


\subsection{Energy estimates}\label{s4.2}

Now,  we aim to look for
 uniform estimates (i.e. being independent of $N$) of the solutions $u^{N}$, which
arrive at the local solvability of \eqref{c-1}--\eqref{c-3}. To this
end, bearing in mind the obtained properties of $c_{j}$,  exploiting
 system \eqref{eq:discretized} and appealing to assumption
\eqref{2.2}, we arrive at
\[
  \frac{1}{2} \frac{d}{dt} \| u^N(\cdot,t)\|^2 +d_{0} \int \limits_{\Omega} |\nabla u^N(x,t) |^2 \, dx
 \leq \|a\|_{L^\infty(\Omega_{\infty})} \|u^N (\cdot,t)\|^2
+\|a\|_{L^\infty(\Omega_{\infty})} \|u^N
(\cdot,t)\|^3_{L^3(\Omega)}.
\]
To manage the last term in the right-hand side, we first apply
interpolation inequality \eqref{g-n-0} with $p=3$ and
$\theta=\frac{d}{6},$ and then the sequential application of Jensen
and Young inequalities (appealing to  $d<4$) provides the estimate
\begin{align*}
 \frac{1}{2} \frac{d}{dt} \| u^N(\cdot,t) \|^2 +  d_{0} \| \nabla u^N (\cdot,t)\|^{2 }
 &\leq
 2c_{1}^{3}\epsilon^{4/d}\|a\|_{L^{\infty}(\Omega_{\infty})}\|\nabla
 u^{N}(\cdot,t)\|^{2}
+\|a\|_{L^{\infty}(\Omega_{\infty})}\Big[
\|u^{N}(\cdot,t\|)\|^{2}
\\
 &
+2c_{1}^{3}(4-d)\epsilon^{-\frac{4}{4-d}}\|u^{N}(\cdot,t)\|^{\frac{2(6-d)}{4-d}}
+8c_{2}^{3}\|u\|^{3} \Big]
\end{align*}
 with some positive $\epsilon$.

 Then, setting in this inequality
 \begin{align*}
 \epsilon&=\bigg(\frac{d_{0}}{4c_{1}^{3}\|a\|_{L^{\infty}(\Omega_{\infty})}}\bigg)^{d/4},\\
C&=2\max\bigg\{\|a\|_{L^{\infty}(\Omega_{\infty})},(4-d)2^{\frac{4-3d}{4-d}}c_{1}^{\frac{12}{4-d}}d_{0}^{-\frac{d}{4-d}}\|a\|_{
L^{\infty}(\Omega_{\infty})}^{\frac{d}{4-d}};
8\|a\|_{L^{\infty}(\Omega_{\infty})}c_{2}^{3} \bigg\},
 \end{align*}
 we end up with the bound
\[
\frac{d}{dt} \| u^N(\cdot,t) \|^2 +   \| \nabla u^N (\cdot,t)\|^{2 }
 \leq
 C[\|u^{N}(\cdot,t)\|^{2}+\|u^{N}(\cdot,t)\|^{3}+\|u^{N}(\cdot,t)\|^{\frac{2(6-d)}{4-d}}].
\]
Finally, thanks to $\frac{2(6-d)}{4-d}>3,$ we obtain the estimate
\[
\frac{d}{dt} \| u^N(\cdot,t) \|^2 +   \| \nabla u^N (\cdot,t)\|^{2 }
 \leq
 3C[1+\|u^{N}(\cdot,t)\|^{\frac{2(6-d)}{4-d}}],
\]
which in turn (after integrating over $[0,\tau]$) gives
\[
 \| u^N(\cdot,\tau) \|^2 +  \int_{0}^{\tau} \| \nabla u^N (\cdot,t)\|^{2
 }dt
 \leq
 3C\tau+ \| u^N_{0} \|^2
 +3C\int_{0}^{\tau}\|u^{N}(\cdot,t)\|^{\frac{2(6-d)}{4-d}}dt.
\]

At this point, exploiting
 the  nonlinear generalization of Gr\"{o}nwall inequality \eqref{2.3}
with $h=T_{N}^{*}$ satisfying  the inequality
\begin{equation}\label{2.4}
3C(T_{N}^{*})^{\frac{6-d}{2}}+(T_{N}^{*})^{\frac{4-d}{2}}\|u_{0}^{N}\|^{2}\leq\bigg(\frac{4-d}{6C}\bigg)^{\frac{4-d}{2}}
\end{equation}
and
\[
c_{3}=3CT_{N}^{*}+\|u_{0}^{N}\|^{2},\quad b(s)=3C,\quad
\gamma=\frac{6-d}{4-d},
\]
 we deduce the bound
\begin{align}\label{2.4**}\notag
\|u^{N}(\cdot,\tau)\|^{2}+\int_{0}^{\tau}\|\nabla
u^{N}(\cdot,t)\|^{2}dt&\leq
[3CT_{N}^{*}+\|u_{0}^{N}\|^{2}]^{\frac{d-2}{2}}\bigg[
3CT_{N}^{*}+\|u_{0}^{N}\|^{2}- \frac{6C\tau}{4-d}
 \bigg]^{\frac{4-d}{2}}\\
 &
< 3CT_{N}^{*}+\|u_{0}^{N}\|^{2}\equiv C_{3}(T_{N}^{*},u_{0}^{N})
\end{align}
for each $\tau\in[0,T_{N}^{*}]$.

\noindent Performing the straightforward calculations tells us that
\begin{equation*}\label{2.4*}
T_{N}^{*}=\min\bigg\{1;\tfrac{4-d}{C}[3C+\|u_{0}^{N}\|^{2}]^{-\frac{2}{4-d}};
\big[\tfrac{4-d}{C}\big]^{\frac{4-d}{6-d}}[3C+\|u_{0}^{N}\|^{2}]^{-\frac{2}{6-d}}
\bigg\}
\end{equation*}
satisfies inequality \eqref{2.4}.

Bearing in mind the strong convergence of $u_{0}^{N}$ to $u_{0}$ in
$L^{2}(\Omega)$ and passing to the limit in the expression of the
$T_{N}$, we get
\[
\underset{N\to+\infty}{\lim}
T_{N}^{*}=\min\bigg\{1;\tfrac{4-d}{C}[3C+\|u_{0}\|^{2}]^{-\frac{2}{4-d}};
\big[\tfrac{4-d}{C}\big]^{\frac{4-d}{6-d}}[3C+\|u_{0}\|^{2}]^{-\frac{2}{6-d}}\bigg\}\equiv
T_{0}.
\]
These relations and the explicit form of
$C_{3}(T_{N}^{*},u_{0}^{N})$ arrive at the bound
\[
C_{3}(T_{N}^{*},u_{0}^{N})\leq C_{3}(T_{0},u_{0})\equiv C_{3}
\]
for all $N>N_{0}>1$.
 Taking $N_0$ greater if necessary, we define the time of
the existence
$$
T^* = \frac{9}{10}T_{0}< T_N^*
$$
for all $N > N_0$.

Summing up, we end up with the estimate
\begin{equation*}\label{2.5}
\|u^{N}(\cdot,\tau)\|^{2}+\int_{0}^{\tau}\|\nabla
u^{N}(\cdot,t)\|^{2}dt\leq  C_{3}(T_{0},u_{0})
\end{equation*}
for each $\tau\in[0,T^{*}]$.

Moreover, this bound together with \eqref{g-n-0} means the existence
of  a constant $C^{*}$ being independent of $N$ and $T_{N}$ such
that
\begin{equation}\label{apr-002}
\|u^{N}\|_{V^{2}(\Omega_{T^{*}})} +
\|u^{N}\|_{L^{2}(0,T^{*};L^{4}(\Omega))} \leqslant C^{*}.
\end{equation}
Thus, to complete the uniform estimates of $u^{N}$, we are left to
evaluate
$\|\partial_{t}u^{N}\|_{L^{2}(0,T^{*};(H^{1}(\Omega))^{*})}$. To
this end, we fix $w\in H^{1}(\Omega)$ having
$\|w\|_{H^{1}(\Omega)}\leq 1$ and decompose
\[
w=v+v_{\perp}\quad\text{with} \quad v\in \Pi_{N}\quad\text{and}\quad
v_{\perp}\in \Pi_{N}^{\perp}
\]
with respect to the scalar product in $L^{2}(\Omega)$. This means
\[
\int_{\Omega}v_{\perp}\phi_{i}dx=0\quad
i=0,1,2,...,N,\quad\text{and}\quad \|v\|_{H^{1}(\Omega)}\leq 1.
\]
Exploiting the equation in \eqref{eq:discretized} and collecting the
properties of the functions $w,v$ with the regularity of
$D(x),a(x,t)$, we immediately conclude that
\begin{align}\label{2.5*}\notag
|\langle\partial_{t}u^{N}(\cdot,t),w\rangle|&\leq
\int_{\Omega}|D(x)||\nabla u^{N}(x,t)||\nabla v(x)|dx +
\int_{\Omega}|a(x,t)||u^{N}(x,t)||v(x)|dx
\\
 &
 +
\int_{\Omega}|a(x,t)||u^{N}(x,t)|^{2}|v(x)|dx \equiv
s_{1}(t)+s_{2}(t),
\end{align}
where we set
\begin{align*}
s_{1}(t)&=
 [\|D\|_{L^{\infty}(\Omega)}+\|a\|_{L^{\infty}(\Omega_{\infty})}]\bigg[
\int_{\Omega}|\nabla u^{N}(x,t)||\nabla v|dx +
\int_{\Omega}|u^{N}(x,t)||v(x)|dx
 \bigg],\\
 s_{2}(t)&=
 \|a\|_{L^{\infty}(\Omega_{\infty})}\int_{\Omega}|u^{N}(x,t)|^{2}|v(x)|dx.
\end{align*}

At this point, we treat each $s_{i}(t)$, separately. Collecting
H\"{o}lder inequality with the boundedness of $H^{1}-$norm of $v$
and \eqref{apr-002} arrives at  the bound
\[
s_{1}(t)\leq
[\|D\|_{L^{\infty}(\Omega)}+\|a\|_{L^{\infty}(\Omega_{\infty})}]\|u^{N}(\cdot,t)\|_{H^{1}(\Omega)}.
\]
Coming to the term $s_{2}(t)$ and employing  H\"{o}lder inequality
yield
\begin{align*}
s_{2}(t)&\leq \|a\|_{L^{\infty}(\Omega_{\infty})}\cdot
\begin{cases}
\|u^{N}(\cdot,t)\|^{2}_{L^{4}(\Omega)}\|v\|\qquad\qquad\text{if}\quad
d=1,\\
\|u^{N}(\cdot,t)\|^{2}_{L^{12/5}(\Omega)}\|v\|_{L^{6}(\Omega)}\quad\text{if}\quad
d=2,3.
\end{cases}\\
& \leq C\cdot \begin{cases}
\|u^{N}(\cdot,t)\|^{2}_{L^{4}(\Omega)}\qquad\text{if}\quad
d=1,\\
\|u^{N}(\cdot,t)\|^{2}_{L^{12/5}(\Omega)}\quad\text{if}\quad d=2,3.
\end{cases}
\end{align*}
Here, to get the last inequality we exploited the property of the
function $v$ and, in the case of $d=2,3,$ we first applied
\eqref{g-n-0}.

Finally,  estimates of $s_{i}(t)$ and  \eqref{2.5*} yield at the
bound
\begin{align*}
\|\partial_{t}u^{N}(\cdot,t)\|_{(H^{1}(\Omega))^{*}}&\leq
[\|D\|_{L^{\infty}(\Omega)}+\|a\|_{L^{\infty}(\Omega_{\infty})}]\|u^{N}(\cdot,t)\|_{H^{1}(\Omega)}\\
& +C \|a\|_{L^{\infty}(\Omega_{\infty})}\cdot
\begin{cases}
\|u^{N}(\cdot,t)\|^{2}_{L^{4}(\Omega)}\qquad\text{if}\quad
d=1,\\
\|u^{N}(\cdot,t)\|^{2}_{L^{12/5}(\Omega)}\quad\text{if}\quad d=2,3.
\end{cases}
\end{align*}

Now, keeping in mind this inequality and \eqref{apr-002}, we  end up
with the inequality
\begin{align*}
\|\partial_{t}u^{N}\|^{2}_{L^{2}(0,T^{*};(H^{1}(\Omega))^{*})}&\leq
T^{*}(C^{*})^{2}[\|D\|_{L^{\infty}(\Omega)}+\|a\|_{L^{\infty}(\Omega_{\infty})}]^{2}\\
&+ C\cdot
\begin{cases}
\int_{0}^{T^{*}}\|u^{N}(\cdot,t)\|^{4}_{L^{4}(\Omega)}dt\qquad\text{if}\quad
d=1,\\
\\
\int_{0}^{T^{*}}\|u^{N}(\cdot,t)\|^{4}_{L^{12/5}(\Omega)}dt\quad\text{if}\quad
d=2,3.
\end{cases}
\end{align*}
To handle the last term, we first utilize the H\"{o}lder inequality
and then apply the parabolic embedding \eqref{2.0} (in the
multi-dimensional case) and \eqref{2.0*} (in the one-dimensional
case). Namely,
\begin{align*}
\int_{0}^{T^{*}}\|u^{N}(\cdot,t)\|^{4}_{L^{4}(\Omega)}dt&\leq
(T^{*})^{\frac{1}{3}}\|u^{N}\|^{4}_{L^{6}(\Omega_{T^{*}})}
\\&
\leq
(|\Omega|T^{*})^{\frac{1}{3}}c^{4}_{0}(1+T^{*}|\Omega|^{-2})^{2/3}\|u^{N}\|^{4}_{V^{2}(\Omega_{T^{*}})}
\qquad\text{if}\quad
d=1,\\
\,
\\
\int_{0}^{T^{*}}\|u^{N}(\cdot,t)\|^{4}_{L^{12/5}(\Omega)}dt& \leq
\begin{cases}
(T^{*})^{2/3}\|u^{N}\|^{4}_{L^{12}(0,T^{*};L^{\frac{12}{5}}(\Omega))}\quad\text{
if}\quad
d=2,\\
\\
(T^{*})^{1/2}\|u^{N}\|^{4}_{L^{8}(0,T^{*};L^{\frac{12}{5}}(\Omega))}\qquad\text{if}\quad
d=3.
\end{cases}
\\
\,
\\
& \leq c^{4}_{0}\cdot
\begin{cases}
(T^{*})^{2/3}(1+T^{*}|\Omega|^{-1})^{1/3}\|u^{N}\|^{4}_{V^{2}(\Omega_{T^{*}})}\qquad\text{if}\quad
d=2,\\
\\
(T^{*})^{1/2}(1+T^{*}|\Omega|^{-2/3})^{1/2}\|u^{N}\|^{4}_{V^{2}(\Omega_{T^{*}})}\quad\text{
if}\quad d=3.
\end{cases}
\end{align*}
In fine, estimate \eqref{apr-002} completes the evaluation of
$\|\partial_{t}u^{N}\|^{2}_{L^{2}(0,T^{*};(H^{1}(\Omega))^{*})}$.

In summary, the obtained uniform bounds allow us to deduce
\begin{equation}\label{2.7}
\|u^{N}\|_{V^{2}(\Omega_{T^{*}})}+\|\partial_{t}u^{N}\|^{2}_{L^{2}(0,T^{*};(H^{1}(\Omega))^{*})}
+\|u^{N}\|_{L^{2}(0,T^{*};L^{4}(\Omega))} \leq C
\end{equation}
with the positive value $C$ depending only on the parameters in the
model \eqref{c-1}-\eqref{c-3}, and the cor\-res\-pon\-ding norms of
the given functions but being independent of $N$.


\subsection{Conclusion of the proof of a local weak solvability}\label{s4.2*}

Here, taking into account the results of Subsections
\ref{s4.1}-\ref{s4.2}, we extract a convergent subsequence and pass
to the limit in the equation in \eqref{eq:discretized}. Namely,
using the Banach-Alaoglu Theorem and estimate \eqref{2.7}, we can
utilize a standard arguments to select a subsequence of $u^{N}$
(which we relabel) such that, for $T^{*}>0$ and $N\to+\infty$,
\begin{align*}
&u^{N}\to u\qquad\quad\text{weakly}\quad\,\text{ in}\quad
L^{2}(0,T^{*};
H^{1}(\Omega)),\\
&u^{N}\to u\qquad\quad\text{strongly}\quad\text{in}\quad L^{2}(\Omega_{T^{*}})\quad\text{and a.e. in}\, \Omega_{T},\\
&u^{N}\to u\qquad\quad\text{weakly-}*\,\text{in}\quad
L^{\infty}(0,T^{*};L^{2}(\Omega)),\\
&\partial_{t}u^{N}\to \partial_{t}u\,\quad\text{weakly}\quad\,\text{
in}\quad L^{2}(0,T^{*};(H^{1}(\Omega))^{*})
\end{align*}
for some $u$ belonging to all the spaces above. Besides, \eqref{2.7}
and \eqref{i.1} arrive at the embedding $u\in
C([0,T^{*}];L^{2}(\Omega))$ and at the estimate \eqref{2.7} for $u$
in place $u^{N}$.

Therefore, the constructed $u$ satisfies the required regularity in
Definition \ref{weak-sol}, along with the initial condition.
Finally, we are left to verify the identity \eqref{ident}. To this
end, we fix $M\geq 1$ and, for every $N>M$ and any test function $v$
with the value in $\Pi_{M}$ ($\Pi_{M}\subset \Pi_{N}$), we have
$$
\int \limits_{\Omega_{T^{*}}} {  \partial_t u^N v\, dx dt}  + \int
\limits_{\Omega_{T^{*}}} { D (x)  \nabla u^N  \nabla v \,dx dt }=
\int \limits_{\Omega_{T^{*}}} { a(x,t)  u^N (1 - u^N) v \,dx dt}.
$$
Appealing to the above convergence and \eqref{2.7} readily implies
that
 $u$  verifies
\begin{equation}\label{eq:wf}
\int \limits_{\Omega_{T^{*}}} {  \partial_t u \, v\, dx dt}  + \int
\limits_{\Omega_{T^{*}}} { D (x)  \nabla u   \nabla v \,dx dt }=
\int \limits_{\Omega_{T^{*}}} { a(x,t)  u (1 - u ) v \,dx dt}
\end{equation}
for every test function $v$ with value in $\Pi_M$ and, hence, for
every test function with value in $\bigcup_{N=1}^{+\infty}\Pi_N$.
Since this union is dense in $H^1(\Omega)$, we have that the above
relations hold for any test function $v\in L^{2} ( 0,T^{*};
H^1(\Omega))$.
 This means that the $u(x,t)$ satisfies the identity (\ref{ident}).
That completes the proof of a local weak solvability of
 \eqref{c-1}--\eqref{c-3}. \qed

\subsection{Nonnegativity of a weak solution}\label{s4.3}
Let $u_{0}(x)\geq 0$ in $\Omega$ and set
\[
w(x,t)=\max\{-u(x,t),0\}.
\]
Multiplying equation \eqref{c-1} by the function $w=w(x,t)$ and
integrating over $\Omega_{\tau}=\Omega\times(0,\tau)$ with arbitrary
$\tau\in[0,T^{*}],$ we arrive at the inequality
\[
\int_{\Omega}w^{2}(x,\tau)dx+2d_{0}\int_{0}^{\tau}\int_{\Omega}|\nabla
w|^{2}dxdt\leq
\|a\|_{L^{\infty}(\Omega_{\infty})}\int_{0}^{\tau}\int_{\Omega}(|w|^{2}+|w|^{3})dxdt.
\]
Here we used the easily verified equality $w(x,0)=0$ a.e. in
$\Omega$.

Next, appealing to Gr\"{o}nwal type inequality \eqref{2.3}, we end
up with the bound
\[
\int_{\Omega}w^{2}(x,\tau)dx\leq 0
\]
for each $\tau\in[0,T^{*}]$, which tells that  $w(x,t)=0$ a.e. in
$\Omega_{T^{*}}$. Coming to the definition of $w(x,t),$ we
immediately conclude that $u(x,t)\geq 0$ a.e. in $\Omega_{T^{*}}$.
That completes the proof of the nonnegativity of a weak solution.
\qed

\subsection{Uniqueness of a weak solution}\label{s4.4}
The arguments exploited to demonstrate the uniqueness of a weak
solution to \eqref{c-1}--\eqref{c-3} is pretty standard. We assume
the existence of two weak solutions $u$ and $\bar{u}$ to
\eqref{c-1}-\eqref{c-3}, which have the regularity obtained in
Subsection \ref{s4.2}, in particulary, $u,\bar{u}\in
C([0,T^{*}],L^{2}(\Omega)).$ Then the difference $U=u-\bar{u}$
satisfies the following relations in the sense of Definition
\ref{weak-sol}
\[
\begin{cases}
U_{t}-\nabla D(x)\nabla U=aU[1-u-\bar{u}]\quad\text{in}\quad
\Omega_{T^{*}},\\
\nabla U\cdot \mathbf{n}=0\qquad\qquad \text{on}\quad
\partial\Omega_{T^{*}},\\
U(x,0)=0\qquad\qquad\text{in}\quad\Omega.
\end{cases}
\]
Multiplying the equation in this system by $U$ and integrating over
$\Omega_{\tau}$ with arbitrary $\tau\in[0,T^{*}]$, we obtain
\[
\frac{1}{2}\int_{\Omega}U^{2}(x,\tau)dx+d_{0}\int_{0}^{\tau}\int_{\Omega}|\nabla
U|^{2}dx dt\leq
\|a\|_{L^{\infty}(\Omega_{\infty})}\int_{0}^{\tau}\int_{\Omega}U^{2}[1+|u|+|\bar{u}|]dxdt.
\]
Here, we used assumption \eqref{2.2} and the regularity of the
coefficient $a(x,t).$

To manage the right-hand side, we apply H\"{o}lder inequality and
the smoothness of $u$ and $\bar{u}.$ Thus, we deduce that
\begin{align*}
&\int_{\Omega}U^{2}(x,\tau)dx+2d_{0}\int_{0}^{\tau}\int_{\Omega}|\nabla
U|^{2}dt dx\\
& \leq \|a\|_{L^{\infty}(\Omega_{\infty})}
[1+\|u\|_{C([0,T^{*}],L^{2}(\Omega))}+\|\bar{u}\|_{C([0,T^{*}],L^{2}(\Omega))}]
\int_{0}^{\tau}\|U\|^{2}_{L^{4}(\Omega)}dt.
\end{align*}
Then, appealing to inequality \eqref{g-n-0} with $p=4$ and
$\theta=d/4,$ and performing technical calculations, we end up with
the estimate
\[
\|U\|^{2}_{L^{4}(\Omega)}\leq \varepsilon\|\nabla
U\|^{2}+C(\varepsilon)\|U\|^{2}
\]
 with some positive $\varepsilon<2d_{0}$, which in turn provides the bound
\[
\int_{\Omega}U^{2}(x,\tau)dx\leq C
\|a\|_{L^{\infty}(\Omega_{\infty})}
[1+\|u\|_{C([0,T^{*}],L^{2}(\Omega))}+\|\bar{u}\|_{C([0,T^{*}],L^{2}(\Omega))}]
\int_{0}^{\tau}\|U\|^{2}_{L^{2}(\Omega)}dt.
\]
In fine, utilizing the Gr\"{o}nwall lemma \cite{Gr} arrives at the
inequality
\[
\int_{\Omega}U^{2}(x,\tau)dx\leq 0
\]
for each $\tau\in[0,T^{*}],$ which finishes the proof of the
uniqueness. \qed

\section{Proof of Theorem \ref{Th-ex1}}\label{s5}

\noindent In the one-dimensional case, collecting Sobolev embedding
theorem (see, e.g. \cite[Section 5.4]{A}) with \cite[Theorem
3.1]{LM} allows one to readily reach the boundedness of a solution
$u$ in $\Omega_{T^{**}}$, exploiting only the regularity of a weak
solution stated in Theorem \ref{Th-ex}. Thus, we are left to verify
Theorem \ref{Th-ex1} in the multi-dimensional case ($d=2,3$). To this
end, bearing in mind the assumption on $u_{0}(x)$ and denoting
\[
k_{0}=\max\{1,\|u_{0}\|_{L^{\infty}(\Omega)}\},
\]
we introduce the function $v_{k}(x,t)$ and the sets $A_{k}(t)$ and
$f_{t}(k)$ for any $k\geq k_{0}$ and $t\leq T^{*}:$
\[
v_{k}=v_{k}(x,t)=\max\{u(x,t)-k,0\},\quad
A_{k}(t)=\{x\in\Omega:\quad u(x,t)\geq k\},\quad
f_{t}(k)=\int_{0}^{t}|A_{k}(\tau)|d\tau.
\]
Performing the straightforward calculations leads to the following
properties of the introduced function and sets.
\begin{corollary}\label{c3.1}
There hold:

\noindent(i)  $A_{k_{2}}(t)\subset A_{k_{1}}(t)$ for any
 $k_{2}>k_{1}\geq k_{0}$ and each $t\in[0,T^{*}];$

\noindent(ii) $f_{t}(k_{0})\leq |\Omega|t$ for each $t\in[0,T^{*}];$

\noindent(iii) $v_{k}(x,0)=0$ for any $x\in\bar{\Omega}$ and
$$\int_{A_{k}(t)}u^{2}(x,t)v_{k}(x,t)dx\leq \int_{A_{k}(t)} [u(x,t)
v_{k}^{2}(x,t)+k u(x,t)v_{k}(x,t)]dx$$
 for each $t\in [0,T^{*}],$ where
$u$ is a weak solution constructed in Theorem \ref{Th-ex}.
\end{corollary}

Clearly,  Theorem \ref{Th-ex1} will follow immediately from the
equality
\begin{equation}\label{3.5}
f_{t}(k)=0
\end{equation}
with some $k>k_{0}$ and any $t\in[0,T^{**}]$, where $T^{**}\leq
T^{*}$ being specified below. Indeed, \eqref{3.5} suggests that
\[
|A_{k}(t)|=0\quad\text{for all}\quad t\in [0,T^{**}],
\]
which in turn gives the desired bound
\begin{equation}\label{3.5*}
\|u\|_{L^{\infty}(\Omega_{T^{**}})}\leq C k_{0}
\end{equation}
and, hence, completes the proof of this theorem.

Thus, we are left  to verify \eqref{3.5}. To this end, we first
obtain Stampacchia type lemma playing the key role in the further
analysis.
\begin{lemma}\label{L-1}
Let $\alpha$ and $\beta$ be positive, and $x_{0}\geq 0$. We assume
that the function $f=f(x):[x_{0},+\infty)\mapsto[0,+\infty)$ is
non-increasing in $[x_{0},+\infty)$ and the following inequality
holds, for any $ y>x\geq x_{0},$
\begin{equation}\label{3.1}
f(y)\leq c_{4}x^{\alpha}(y-x)^{-\alpha}f^{\beta}(x)
\end{equation}
with a positive $c_{4}.$

If
\begin{equation}\label{3.1*}
\beta>1 \quad\text{and}\quad
0<c_{4}f^{\beta-1}(x_{0})2^{\frac{\alpha\beta}{\beta-1}}<1,
\end{equation}
then
\[
f(y)=0
\]
for each $y\geq
x_{0}\big[1-2^{\frac{\beta}{\beta-1}}c^{1/\alpha}_{4}f^{(\beta-1)/\alpha}(x_{0})\big]^{-1}.$
\end{lemma}
\begin{proof}
For  positive $\alpha,\beta$ (in \eqref{3.1}) and $x_{0},$ we
introduce the function
\[
\mathcal{F}=\mathcal{F}(y)=\bigg(\frac{f(y)}{f(x_{0})}\bigg)^{1/\alpha}
\]
with $y\geq x_{0}$.

\noindent Thanks to the properties of the function $f,$ we
immediately conclude that $\mathcal{F}$ is well-defined for each
$y\geq x_{0}$ and $\mathcal{F}(x_{0})=1.$

Clearly, in order to prove this lemma, it is enough to find some $b$
satisfying inequalities
\begin{equation}\label{3.2}
x_{0}<b<x_{0}\big[1-2^{\frac{\beta}{\beta-1}}c^{1/\alpha}_{4}f^{(\beta-1)/\alpha}(x_{0})\big]^{-1}
\end{equation}
such that
\begin{equation}\label{3.3}
\mathcal{F}(b+x_{0})=0.
\end{equation}
Indeed, collecting this equality with the properties of
$\mathcal{F}(y)$ provides the desired equality in Lemma \ref{L-1}.
Thus, we need to verify \eqref{3.3} with $b$ satisfying \eqref{3.2}.
To this end, we rewrite the inequality \eqref{3.1} in the term of
$\mathcal{F}$ and get
\begin{equation}\label{3.4}
\mathcal{F}(y)\leq\frac{C_{4}x}{y-x}\mathcal{F}^{\beta}(x)
\end{equation}
with $ x_{0}\leq x<y$ and
$C_{4}=c_{4}^{1/\alpha}f^{\frac{\beta-1}{\alpha}}(x_{0})$.

\noindent For some $b$  satisfying \eqref{3.2}, which will be
specified below, we construct the convergent subsequence
\[
x_{n}=x_{0}+b(1-2^{-n}),\quad n=0,1,2,....
\]
It is apparent that
\[
x_{n+1}-x_{n}=b2^{-n-1},\quad  x_{n}\leq x_{0}+b\quad\text{for
all}\quad n\in\mathbb{N}, \quad \text{and}\quad
x_{n}\underset{n\to+\infty}{\rightarrow}x_{0}+b.
\]
Then substituting $x=x_{n}$ and $y=x_{n-1}$ to \eqref{3.4} arrives
at the bound
\[
\mathcal{F}(x_{n+1})\leq
C_{4}\frac{x_{0}+b}{b}2^{n+1}\mathcal{F}^{\beta}(x_{n}).
\]
Arguing by induction, we derive
\[
\mathcal{F}(x_{n})\leq
C_{4}\frac{x_{0}+b}{b}2^{n}\mathcal{F}^{\beta}(x_{n-1})\leq
\bigg[C_{4}\frac{x_{0}+b}{b}\bigg]^{\sum_{k=0}^{n-1}\beta^{k}}2^{S_{n}},
\]
where
\[
S_{n}=\sum_{k=0}^{n-1}(n-k)\beta^{k}.
\]
Applying the following easily verified  equalities
\[
\sum_{k=0}^{n-1}z^{k}=\frac{z^{n}-1}{z-1},\quad
\sum_{k=0}^{n-1}kz^{k-1}=\frac{nz^{n-1}(z-1)+1-z^{n}}{(z-1)^{2}}
\]
for $z>1$, we compute
\[
\sum_{k=0}^{n-1}\beta^{k}=\frac{\beta^{n}-1}{\beta-1},\quad
S_{n}=n\sum_{k=0}^{n-1}\beta^{k}-\beta\sum_{k=0}^{n-1}k\beta^{k-1}=\frac{\beta^{n+1}-(n+1)\beta+n}{(\beta-1)^{2}}.
\]
Exploiting these relations to evaluate $\mathcal{F}(x_{n})$ arrives
at
\[
\mathcal{F}(x_{n})\leq
\bigg(C_{4}\frac{x_{0}+b}{b}\bigg)^{\frac{\beta^{n}-1}{\beta-1}}
2^{\frac{\beta^{n+1}-(n+1)\beta+n}{(\beta-1)^{2}}}.
\]
Taking into account condition \eqref{3.1*}, we select $b$ solving
the equation
\[
\frac{x_{0}+b}{b}C_{4}=2^{\frac{-\beta}{\beta-1}},
\]
that is
\[
b=\frac{C_{4}x_{0}2^{\frac{\beta}{\beta-1}}}{1-2^{\frac{\beta}{\beta-1}}
C_{4}}.
\]
After that, we draw up
\[
\mathcal{F}(x_{n})\leq 2^{-\frac{n}{\beta-1}}.
\]
Since $x_{n}\leq x_{0}+b$ and $\mathcal{F}$ is non-increasing, the
last inequality ensures
\[
0\leq \mathcal{F}(x_{0}+b)\leq 2^{-\frac{n}{\beta-1}}.
\]
Passing to the limit in this estimates as $n\to+\infty$ arrives at
\eqref{3.3}, which finishes the proof of this lemma.
\end{proof}
Now, taking into account Theorem \ref{Th-ex} and Corollary
\ref{c3.1}, we multiply equation \eqref{c-1} by $v_{k}$ and,
integrating over $\Omega_{\tau},$ $\tau< T^{*},$ and performing
standard technical calculations,  we have
\begin{equation}\label{3.6}
\frac{1}{2}\int_{\Omega}v_{k}^{2}(x,\tau)dx+d_{0}\int_{0}^{\tau}\int_{\Omega}|\nabla
 v_{k}(x,t)|^{2}dxdt\leq i_{1}+i_{2},
\end{equation}
where we set
\begin{align*}
i_{1}&=\|a\|_{L^{\infty}(\Omega_{\infty})}\int_{0}^{\tau}\int_{A_{k}(t)}u(x,t)v_{k}^{2}(x,t)dxdt,\\
i_{2}&=2k\|a\|_{L^{\infty}(\Omega_{\infty})}\int_{0}^{\tau}\int_{A_{k}(t)}u(x,t)v_{k}(x,t)dxdt.
\end{align*}
At this point, we evaluate each $i_{j},$ separately.

\noindent$\bullet$ Exploiting the Cauchy-Schwarz inequality, we
arrive at
\[
\int_{A_{k}(t)}u(x,t)v_{k}^{2}(x,t)dx\leq
\|v_{k}(\cdot,t)\|^{2}_{L^{4}(\Omega)}\|u(\cdot,t)\|_{L^{2}(A_{k}(t))}.
\]
Applying \eqref{g-n-0} and then the Jensen inequality to handle the
term $\|v_{k}(\cdot,t)\|^{2}_{L^{4}(\Omega)}$ yields the bound
\begin{align*}
\int_{A_{k}(t)}u(x,t)v_{k}^{2}(x,t)dx&\leq\frac{d_{0}}{4\|a\|_{L^{\infty}(\Omega_{\infty})}}\|\nabla
v_{k}(\cdot,t)\|^{2}\\
&+C_{5}\|v_{k}(\cdot,t)\|^{2}[\|u(\cdot,t)\|+\|u(\cdot,t)\|^{\frac{4}{4-d}}],
\end{align*}
where the positive quantity $C_{5}$ is independent of $k$, $\tau$
and $|A_{k}(t)|$.

\noindent Collecting this estimate and performing technical
calculations, we end up with the inequality
\[
i_{1}\leq \frac{d_{0}}{4}\|\nabla
v_{k}\|^{2}_{L^{2}(\Omega_{\tau})}+C_{5}\|a\|_{L^{\infty}(\Omega_{\infty})}\int_{0}^{\tau}\|v_{k}(\cdot,t)\|^{2}[\|u(\cdot,t)\|
+\|u(\cdot,t)\|^{\frac{4}{4-d}}]dt.
\]
Finally, bearing in mind the regularity of $u$, i.e. $u\in
C([0,T^{*}],L^{2}(\Omega))$ (see \eqref{apr-002}), we complete the
evaluation of $i_{1}$ with the bound
\[
i_{1}\leq \frac{d_{0}}{4}\|\nabla
v_{k}\|^{2}_{L^{2}(\Omega_{\tau})}+C_{5}[C^{*}+C^{*\frac{4}{4-d}}]\int_{0}^{\tau}\|v_{k}(\cdot,t)\|^{2}dt\quad
\forall\tau\in[0,T^{*}].
\]

\noindent$\bullet$ Coming to  $i_{2}$, we first analyze the case of
$d=3$. Utilizing sequentially the H\"{o}lder inequality with
exponents $p=6$ and $q=6/5$, and then \eqref{g-n-0} with $p=6$ and
$\theta=1$, we obtain
\begin{align*}
2k\|a\|_{L^{\infty}(\Omega_{\infty})}\int_{A_{k}(t)}u(x,t)v_{k}(x,t)dx&\leq
2k\|a\|_{L^{\infty}(\Omega_{\infty})}\|v_{k}(\cdot,t)\|_{L^{6}(\Omega)}\|u(\cdot,t)\|_{L^{6/5}(A_{k}(t))}\\
& \leq 2c_{1}k\|a\|_{L^{\infty}(\Omega_{\infty})}\|\nabla
v_{k}(\cdot,t)\|\|u(\cdot,t)\|_{L^{6/5}(A_{k}(t))}\\
&+2c_{2}k\|a\|_{L^{\infty}(\Omega_{\infty})}\|v_{k}(\cdot,t)\|.
\end{align*}
At last, applying Cauchy-Schwarz inequality provides the bound
\begin{align}\label{3.7}\notag
2k\|a\|_{L^{\infty}(\Omega_{\infty})}\int_{A_{k}(t)}u(x,t)v_{k}(x,t)dx&\leq
4k^{2}[c_{1}^{2}d_{0}^{-1}+c_{2}^{2}]\|a\|^{2}_{L^{\infty}(\Omega_{\infty})}\|u(\cdot,t)
\|^{2}_{L^{6/5}(A_{k}(t))}
\\
&+ \frac{d_{0}}{4}\|\nabla
v_{k}(\cdot,t)\|^{2}+\frac{1}{2}\|v_{k}(\cdot,t)\|^{2}
\quad\text{if}\quad d=3.
\end{align}

In the two-dimensional case, exploiting the similar technique
(utilizing \eqref{g-n-0} and Young inequality with the corresponding
exponents) arrives at the estimate
\begin{align*}
&k\|a\|_{L^{\infty}(\Omega_{\infty})}\int\limits_{A_{k}(t)}u(x,t)v_{k}(x,t)dx\\&\leq
4k^{2}\|a\|_{L^{\infty}(\Omega_{\infty})}\bigg[
\bigg(c_{2}+\frac{2c_{1}}{p_{1}}\bigg)^{2}+\frac{c_{1}^{2}(p_{1}-2)^{2}}{d_{0}p_{1}^{2}}\bigg]\|u(\cdot,t)
\|^{2}_{L^{\frac{p_{1}}{p_{1}-1}}(A_{k}(t))}
\\
&+ \frac{d_{0}}{8}\|\nabla
v_{k}(\cdot,t)\|^{2}+\frac{1}{2}\|v_{k}(\cdot,t)\|^{2}
\end{align*}
with some $p_{1}>2$ being specified below.

In fine, taking into account this estimate and \eqref{3.7}, we
complete the evaluation of $i_{2}:$
\[
i_{2}\leq \frac{d_{0}}{4}\int_{0}^{\tau}\|\nabla
v_{k}(\cdot,t)\|^{2}dt+\int_{0}^{\tau}\|v_{k}(\cdot,t)\|^{2}dt+k^{2}C_{6}\int_{0}^{\tau}\|u(\cdot,t)\|^{2}_{L^{q_{1}}(A_{k}(t))}dt,
\]
where
\[
q_{1}=\begin{cases} 6/5\quad\text{ if}\quad d=3,\\
\frac{p_{1}}{p_{1}-1}\quad\text{if}\quad d=2,
\end{cases}
\]
and
\[
C_{6}=\|a\|^{2}_{L^{\infty}(\Omega_{\infty})}\begin{cases}
2(c_{1}^{2}d_{0}^{-1}+c_{2})\qquad\qquad\qquad\quad \text{if}\quad d=3,\\
8 \bigg[
\bigg(c_{2}+\frac{2c_{1}}{p_{1}}\bigg)^{2}+\frac{c_{1}^{2}(p_{1}-2)^{2}}{d_{0}p_{1}^{2}}\bigg]\quad
\text{if}\quad d=2.
\end{cases}
\]
Finally, to manage the last term in the right-hand side, we utilize
the H\"{o}lder inequality and get the bound
\[
\|u(\cdot,t)\|^{2}_{L^{q_{1}}(A_{k}(t))}\leq
\|u(\cdot,t)\|^{2}_{L^{2}(\Omega)}|A_{k}(t)|^{q_{2}}\quad\text{with}\quad
q_{2}=\begin{cases} 2/3\quad\text{ if}\quad d=3,\\
\frac{p_{1}-2}{p_{1}}\quad\text{if}\quad d=2,
\end{cases}
\]
which in turn provides
\[
i_{2}\leq \frac{d_{0}}{4}\int_{0}^{\tau}\|\nabla
v_{k}(\cdot,t)\|^{2}dt+\int_{0}^{\tau}\|v_{k}(\cdot,t)\|^{2}dt +
C_{6}k^{2}\|u\|^{2}_{C([0,T^{*}],L^{2}(\Omega))}\int_{0}^{\tau}|A_{k}(t)|^{q_{2}}dt.
\]
Here, to obtain the last term in the right-hand side, we used the
regularity of a weak solution established in Theorem \ref{Th-ex}.

At this point, coming to \eqref{3.6} and collecting estimates of
$i_{1}$ and $i_{2}$, we deduce
\[
\int_{\Omega}v_{k}^{2}(x,\tau)dx+d_{0}\int_{0}^{\tau}\int_{\Omega}|\nabla
v_{k}(x,t)|^{2}dxdt\leq
C_{7}\|v_{k}\|^{2}_{L^{2}(\Omega_{\tau})}+2C_{6}\|u\|^{2}_{C([0,T^{*}],L^{2}(\Omega))}k^{2}\int_{0}^{\tau}|A_{k}(t)|^{q_{2}}dt.
\]
At last, the Gr\"{o}nwall Lemma \cite{Gr} produces the bound
\[
\|v_{k}\|_{L^{\infty}(0,\tau;L^{2}(\Omega))}+\|v_{k}\|_{L^{2}(0,\tau;H^{1}(\Omega))}\leq
C_{8}k\bigg(\int_{0}^{\tau}|A_{k}(t)|^{q_{2}}dt\bigg)^{1/2}\quad\forall
\tau\in(0,T^{*}]
\]
with positive $C_{8}$ being independent of $\tau$ and $k$.

\noindent Utilizing the compact embedding \eqref{2.0*} with
$q=\frac{2(d+2)}{d},$ $p=2$ to manage the left-hand side of the last
inequality and applying Young inequality with exponents
$\frac{1}{q_{2}}$ and $\frac{1}{1-q_{2}}$ to handle the right-hand
side lead to the estimate
\begin{equation}\label{3.9}
\|v_{k}\|_{L^{\frac{2(d+2)}{d}}(\Omega_{\tau})}\leq
C_{9}k(f_{\tau}(k))^{q_{2}/2} \quad\forall \tau\in(0,T^{*}]
\end{equation}
with $C_{9}$ being independent of $\tau$ and $k$ and depending on
$T^{*}$.

On the other hand, employing Corollary \ref{c3.1} gives
\begin{equation}\label{3.10*}
\|v_{k}\|_{L^{\frac{2(d+2)}{d}}(\Omega_{\tau})}\geq(k_{2}-k)(f_{\tau}(k_{2}))^{\frac{d}{2(d+2)}}
\end{equation}
for any $k_{2}>k\geq k_{0}.$ Collecting this inequality with
\eqref{3.9} entails
\[
f_{\tau}(k_{2})\leq
\bigg(C_{9}\frac{k}{k_{2}-k}\bigg)^{\frac{2(d+2)}{d}}(f_{\tau}(k))^{\frac{q_{2}(d+2)}{d}}.
\]
Obviously, $\frac{q_{2}(d+2)}{d}>1$ if either $d=3$ or $d=2$ and
$p_{1}>4$. Hence, applying Lemma \ref{L-1} with
\[
\beta=\frac{q_{2}(d+2)}{d},\quad \alpha=\frac{2(d+2)}{d},\quad
x=k,\quad y=k_{2},\quad c_{4}=C_{9}^{\frac{2d+4}{d}}
\]
ends up with the equality
\[
f_{\tau}(k_{2})=0
\]
for each $k_{2}\geq
k_{0}[1-2^{\frac{q_{2}(d+2)}{q_{2}(d+2)-d}}C_{9}(f_{\tau}(k_{0}))^{\frac{q_{2}(d+2)-d}{2d+4}}]^{-1}$
provided
\begin{equation}\label{3.10}
C_{9}^{\frac{2d+4}{d}}(f_{\tau}(k_{0}))^{\frac{q_{2}(d+2)-d}{d}}2^{\frac{2q_{2}(d+2)^{2}}{d(q_{2}(d+2)-d)}}<1.
\end{equation}
Corollary \ref{c3.1} tells that inequality \eqref{3.10} holds if
\[
\tau<\min\{T^{*},
C_{9}^{-\frac{2d+4}{q_{2}(d+2)-d}}2^{-2q_{2}(\frac{d+2}{q_{2}(d+2)-d})^{2}}\}\equiv
T_{1}
\]
In fine, selecting
\[
T^{**}=\min\{T_{1},|\Omega|^{-1}
C_{9}^{-\frac{2d+4}{q_{2}(d+2)-d}}2^{\frac{-2(d+2)[q_{2}(d+3)-d]}{(q_{2}(d+2)-d)^{2}}}
\},
\]
we conclude that
\[
f_{\tau}(k_{2})=0\qquad\text{for any}\quad \tau\in[0,T^{**}]\quad
\text{and}\quad k_{2}\geq 2k_{0},
\]
which leads to \eqref{3.5*} with $C=2$. This finishes the proof of
Theorem \ref{Th-ex1}. \qed


\section{The proof of Theorem~\ref{Th-2}}\label{s6}

\noindent The requirements in this claim are stronger than them in
Theorem \ref{Th-ex}, thus, we can immediately conclude with the
unique nonnegative weak solution to \eqref{c-1}--\eqref{c-3}.

After that, we carry out the arguments in two stages. Assuming
\eqref{i.2}, the first step is related with the verification of the
relations
\begin{equation}\label{6.0}
u(x,t)\leq 1 \quad\text{a.e. in}\quad
\Omega_{T^{*}}\quad\text{and}\quad \ln(1-u)\in
L^{2}(0,T^{*};H^{1}(\Omega)).
\end{equation}
To this end, for any positive $\varepsilon$, we introduce  the
entropy function $G_{\varepsilon}(z),$  and consider the approximate
problems constructed via the approximated sequence of the initial
data. In the second stage, we examine the regularity $\ln u\in
L^{2}(0,T^{*};H^{1}(\Omega))$ and obtain the second estimate in
\eqref{entr-e}.

\noindent\textit{Step I:} At this point, for each $\varepsilon\geq
0$, we define the function
\[
G_{\varepsilon}(z)=\begin{cases}
\frac{z-1}{\sqrt{\varepsilon}}\arctan\frac{\sqrt{\varepsilon}z}{1-z+\varepsilon}
-\frac{1}{2}\ln\frac{\varepsilon+(z-1)^{2}}{1+\varepsilon},\quad\varepsilon>0,\\
-z-\ln|1-z|,\quad \varepsilon=0
\end{cases}
\]
for all $z\geq 0$.

The straightforward calculations provide the following properties of
the function $G_{\varepsilon}(z)$.
\begin{corollary}\label{c7.1}
The following relations hold for any $\varepsilon\geq 0$:

\noindent (i) $G'_{\varepsilon}(0)=G_{\varepsilon}(0)=0;$

\smallskip
\noindent (ii) $G_{\varepsilon}(z)\geq (2\varepsilon)^{-1}(z-1)^{2}$
if $z>1$,

$|G_{\varepsilon}(z)-G_{0}(z)|\leq\frac{\sqrt{\varepsilon}}{4|1-z|}$
if $|z|<1$;

\smallskip
\noindent (iii) $0\leq G''_{0}(z)-G''_{\varepsilon}(z)\leq\frac{
\sqrt{\varepsilon}}{2|1-z|^{3}}$ for any $z\in\mathbb{R}$,

 $G''_{\varepsilon}(z)=\frac{1}{\varepsilon+(1-z)^{2}}$ for
any $z\in \mathbb{R}$;

\smallskip
\noindent (iv) $|G'_{\varepsilon}(z)|\leq
C|z||1-z+\varepsilon|^{-1}$ for any non-negative $z$ and each
positive $\varepsilon$, where
\[
C=\begin{cases} \pi/2\quad\text{if}\quad
z\in\bigg[\frac{1+\varepsilon}{1+\sqrt{\varepsilon}},\frac{1+\varepsilon}{1-\sqrt{\varepsilon}}\bigg],\\
1\qquad\text{if}\quad z\in
\bigg[0,\frac{1+\varepsilon}{1+\sqrt{\varepsilon}}\bigg)\cup\bigg(\frac{1+\varepsilon}{1-\sqrt{\varepsilon}},+\infty\bigg).
\end{cases}
\]
\end{corollary}
Bearing in mind requirement \eqref{i.2}, we approximate the initial
data $u_{0}$ by the sequence
$\{u_{0}^{\varepsilon}\}_{\varepsilon>0}$ having the properties:
\begin{align}\label{6.2}\notag
&u_{0}^{\varepsilon}\in L^{2}(\Omega), \quad 0\leq
u_{0}^{\varepsilon}\leq 1-\varepsilon^{\delta}\quad\text{with
some}\quad \delta\in(0,1/2);\\
& u_{0}^{\varepsilon}\underset{\varepsilon\to 0}{\rightarrow}
u_{0}\quad\text{strongly in}\quad L^{2}(\Omega).
\end{align}
Then, for each $u_{0}^{\varepsilon},$ we consider the approximating
problem
\begin{equation}\label{6.1}
\begin{cases}
u_{t}^{\varepsilon}-\nabla(D(x)\nabla
u^{\varepsilon})=a(x,t)u^{\varepsilon}(1-u^{\varepsilon})\quad
\text{in}\quad \Omega_{T^{*}},\\
\nabla u^{\varepsilon}\cdot \mathbf{n}=0\qquad\qquad\text{on}\quad
\partial\Omega_{T^{*}},\\
u^{\varepsilon}(x,0)=u_{0}^{\varepsilon}(x)\qquad\text{in
}\quad\Omega.
\end{cases}
\end{equation}
Clearly, all requirements of Theorem \ref{Th-ex} are satisfied in
the case of problem \eqref{6.1}, and, hence, Theorem \ref{Th-ex}
provides the one-valued weak solvability of \eqref{6.1} and,
besides, $u^{\varepsilon}\geq 0$ a.e. in $\Omega_{T^{*}}$. Arguing
analogously to the proof of Theorem  \ref{Th-ex} deduces the
corresponding compactness on $\{u^{\varepsilon}\}_{\varepsilon>0}$
and, hence,
\begin{equation}\label{6.2*}
u_{\varepsilon}\underset{\varepsilon\to 0}{\rightarrow}
u\quad\text{strongly in}\quad L^{2}(\Omega_{T^{*}})\quad\text{and
a.e. in} \quad \Omega_{T^{*}},
\end{equation}
where $u$ is a nonnegative weak solution of \eqref{c-1}-\eqref{c-3}
(in the sense of Definition \ref{weak-sol}).

Now, we aim to verify the inequality in \eqref{6.0}. To this end,
arguing by contradiction, we assume that
\[
u(x,t)> 1\quad\text{a.e. in}\quad \Omega_{T^{*}}
\]
and utilize the first estimate in (ii) of Corollary \ref{c7.1} with
$z=u^{\varepsilon}-1$. As a result, we have
\[
\int_{\Omega}(u^{\varepsilon}-1)^{2}dx\leq
2\varepsilon\int_{\Omega}G_{\varepsilon}(u^{\varepsilon})dx.
\]
It is apparent that if the integral in the right-hand side is
uniformly bounded, then passing to the limit in this estimate as
$\varepsilon \to 0$ provides the contradiction with the assumption.
In summary, we conclude that
\[
u(x,t)\leq 1\quad\text{a.e. in}\quad \Omega_{T^{*}}.
\]

Thus, we are left to verify that
\[
\int_{\Omega}G_{\varepsilon}(u^{\varepsilon})dx\leq C
\]
with the positive $C$ being independent of $\varepsilon.$ To this
end, we multiply the equation in \eqref{6.1} by
$G'_{\varepsilon}(u^{\varepsilon})$ and, integrating over
$\Omega_{\tau}$ (after performing technical calculations and
appealing to (iv) in Corollary \ref{c7.1}), we get
\begin{align}\label{6.4*}\notag
\int_{\Omega}G_{\varepsilon}(u^{\varepsilon})dx+\int_{\Omega_{\tau}}D(x)G''_{\varepsilon}(u^{\varepsilon})|\nabla
u^{\varepsilon}|^{2}dxdt&=
\int_{\Omega}G_{\varepsilon}(u^{\varepsilon}_{0})dx+\int_{\Omega_{\tau}}a(x,t)u^{\varepsilon}(1-u^{\varepsilon})G'_{\varepsilon}(u^{\varepsilon})dxdt\\
& \leq I_{1}+I_{2},
\end{align}
where we set
\[
I_{1}=\int_{\Omega}G_{\varepsilon}(u_{0}^{\varepsilon})dx\quad\text{and}\quad
I_{2}=\frac{\pi}{2}\|a\|_{L^{\infty}(\Omega_{\infty})}\int_{\Omega_{\tau}}(u^{\varepsilon})^{2}dxdt.
\]
Clearly, the estimate of $I_{2}$ is a simple consequence of
\eqref{6.2*}, that is
\[
I_{2}\leq C T^{*}
\]
with $C$ being independent of $\varepsilon.$

Coming to $I_{1},$ we rewrite this term in more suitable form
\[
I_{1}=\int_{\Omega}[G_{\varepsilon}(u_{0}^{\varepsilon})-G_{0}(u_{0}^{\varepsilon})]dx+\int_{\Omega}G_{0}(u_{0}^{\varepsilon})dx.
\]
Bearing in mind \eqref{6.2} and applying the second inequality in
(ii) of Corollary \ref{c7.1}, we arrive at
\[
I_{1}\leq
\frac{\varepsilon^{\frac{1}{2}-\delta}}{4}+\int_{\Omega}G_{0}(u_{0}^{\varepsilon})dx.
\]
Finally, the strong convergence $u_{0}^{\varepsilon}$ to $u_{0}$ in
$L^{2}(\Omega)$ yields
\[
 I_{1}\leq
\frac{\varepsilon^{\frac{1}{2}-\delta}}{4}+\int_{\Omega}G_{0}(u_{0})dx.
\]

\noindent Collecting estimates of $I_{j}$ with \eqref{6.4*} entails
\begin{equation}\label{6.5}
\int_{\Omega}G_{\varepsilon}(u^{\varepsilon})dx+d_{0}\int_{\Omega_{\tau}}G''_{\varepsilon}(u^{\varepsilon})|\nabla
u^{\varepsilon}|^{2}dxdt\leq
\frac{\varepsilon^{\frac{1}{2}-\delta}}{4}+\int_{\Omega}G_{0}(u_{0})dx+CT^{*}.
\end{equation}
This inequality together with \eqref{i.2} ensures the uniformly
boundedness of $\int_{\Omega}G_{\varepsilon}(u^{\varepsilon})dx$ for
any $\tau\in[0,T^{*}]$.

In light of  \eqref{i.2}, \eqref{6.2},  and Fatou's Lemma, we deduce
from \eqref{6.5} that
\[
\int_{\Omega}G_{0}(u)dx+d_{0}\int_{\Omega_{\tau}}|\nabla\ln(1-u)|^{2}dxdt\leq
\int_{\Omega}G_{0}(u_{0})dx+CT^{*}
\]
for each $\tau\in[0,T^{*}]$.

\noindent\textit{Step II.} Here, we exploit the modified (in light)
arguments of Step I. Namely, we replace $u_{0}^{\varepsilon}$ and
$G_{\varepsilon}(z)$ by $v_{0}^{\varepsilon}$ and
$\mathcal{G}_{\varepsilon}(z)$, where $v_{0}^{\varepsilon}$
satisfies the relations:
\begin{align*}
&v_{0}^{\varepsilon}\in L^{2}(\Omega), \quad v_{0}^{\varepsilon}\geq
\varepsilon^{\delta}\quad\text{with
some}\quad \delta\in(0,1/2);\\
&v_{0}^{\varepsilon}\underset{\varepsilon\to 0}{\rightarrow}
u_{0}\quad\text{strongly in}\quad L^{2}(\Omega),
\end{align*}
while
\[
\mathcal{G}_{\varepsilon}(z)=
\begin{cases}
z\varepsilon^{-1/2}\arctan\frac{(z-1)\varepsilon^{1/2}}{\varepsilon+z}-\frac{1}{2}\ln\frac{z^{2}+\varepsilon}{1+\varepsilon},\quad\varepsilon>0,
\\
z-1-\ln z,\qquad\qquad \varepsilon=0,
\end{cases}
\]
for each $z\in[0,1]$.

Obviously, $\mathcal{G}_{\varepsilon}(z)$ has the  properties
similar to them established in Corollary \ref{c7.1}. In particular,
\[
\mathcal{G}''_{\varepsilon}(z)=(z^{2}+\varepsilon)^{-1}>0,\quad
\mathcal{G}'_{\varepsilon}(1)=\mathcal{G}_{\varepsilon}(1)=0.
\]
After that, recasting arguments of Step I, we end up with the bound
\[
\int_{\Omega}[u-1-\ln u]dx+d_{0}\int_{\Omega_{\tau}}|\nabla\ln
u|^{2}dxdt \leq CT^{*}+\int_{\Omega}\mathcal{G}_{0}(u_{0})dx,
\]
which completes the proof of Step II and, hence the proof of Theorem
\ref{Th-2}. \qed

\section{Proof of Theorem \ref{t3}}\label{s7}

\noindent It is worth noting that the proof of Theorem \ref{Th-2}
(Section \ref{s6}) as well as the verification of the uniqueness and
the nonnegativity of a weak solution given in Sections \ref{s4.3}
and \ref{s4.4} are true for any  positive fixed $T$ provided $u$ is
a weak solution defined in $\Omega_{T}$. Thus, to verify Theorem
\ref{t3}, we first have to extend the nonnegative weak solution
constructed in Theorem \ref{Th-ex} from the time interval
$[0,T^{*}]$ to the segment $[0,T]$ for any $T>T^{*}$. On the second
stage, we should examine the boundedness of this solution for any
fixed $T>0$.

\noindent\textit{Step 1: Global weak solvability of
\eqref{c-1}--\eqref{c-3}.} Let $u$ be a nonnegative weak solution
built in Theorem \ref{Th-ex}. To extend this solution from
$[0,T^{*}]$ to $[0,T]$ is enough to obtain inequality
\eqref{apr-002} for $u$ instead  of $u^{N}$ with the constant
$C^{*}$ being independent of $T^{*}$ but being controlled with $T$.
After that the standard so-called extension arguments (see, for
instance, the proof of \cite[Theorem 3.3]{CGM}) provides the
existence of a weak solution on whole time interval.

Now, we wish to examine the estimate
\begin{equation}\label{4.1}
\|u\|_{V^{2}(\Omega_{\tau})}\leq C\quad\text{for any}\quad
\tau\in[0,T^{*}],
\end{equation}
where $C$ depends only on the corresponding norm of
$u_{0}(x),a(x,t),$ the Lebesgue measure of $\Omega$ and $ T,$ but is
independent of $T^{*}$.

To refine estimate \eqref{apr-002}, bearing in mind the
nonnegativity of $u(x,t)$ and $a(x,t)$ and testing \eqref{c-1} by
$u$, we deduce that
\begin{align*}
&\int_{\Omega}u^{2}(x,\tau)dx+d_{0}\int_{0}^{\tau}\int_{\Omega}|\nabla
u(x,t)|^{2}dxdt+\int_{0}^{\tau}\int_{\Omega}
a(x,t)u^{3}(x,t)dxdt\\
&\leq \int_{\Omega}u_{0}^{2}(x)dx+\int_{0}^{\tau}\int_{\Omega}
a(x,t)u^{2}(x,t)dxdt
\end{align*}
for any $\tau\in[0,T^{*}].$ To handle the last term in the
right-hand side, we apply the Young inequality with exponents $3/2$
and $3$ and arrive at the bound
\begin{align}\label{4.3}\notag
&\int_{\Omega}u^{2}(x,\tau)dx+d_{0}\int_{0}^{\tau}\int_{\Omega}|\nabla
u(x,t)|^{2}dxdt+\int_{0}^{\tau}\int_{\Omega}a(x,t)u^{3}(x,t)dxdt\\
& \leq
\int_{\Omega}u_{0}^{2}(x)dx+\frac{1}{3}\int_{0}^{\tau}\int_{\Omega}
a(x,t)dxdt+\frac{2}{3}\int_{0}^{\tau}\int_{\Omega}
a(x,t)u^{3}(x,t)dxdt,
\end{align}
which immediately yields the desired estimate
\[
\|u\|_{V^{2}(\Omega_{\tau})}\leq
\|u_{0}\|_{L^{2}(\Omega)}+\|a\|^{1/2}_{L^{\infty}(\Omega_{\infty})}|\Omega|T
\]
for any $\tau\in[0,T^{*}]$. This completes the proof of a global
weak solvability of \eqref{c-1}--\eqref{c-3}.

\noindent\textit{Step 2: The boundedness of $u(x,t)$.} In order to
prove \eqref{3.5*} for any fixed positive $T$, we recast the
arguments of Section \ref{s5}, where instead of Lemma \ref{L-1} we
utilize the classical Stampacchia Lemma \cite[Lemma 4.1]{S1}. To
this end, we modify slightly  the evaluation of nonlinear term in
\eqref{3.6}. Namely, multiplying \eqref{c-1} by $v_{k}$ and then
integrating over $\Omega_{\tau}$ (taking into account the
nonnegativity of $a(x,t)$), we obtain
\begin{align*}
&\frac{1}{2}\int_{\Omega}v_{k}^{2}(x,\tau)dx+d_{0}\int_{0}^{\tau}\int_{\Omega}|\nabla
v_{k}|^{2}dxdt+\int_{0}^{\tau}\int_{A_{k}(t)}a(x,t)u^{2}(x,t)v_{k}(x,t)dxdt\\
&= \int_{0}^{\tau}\int_{A_{k}(t)}a(x,t)u(x,t)v_{k}(x,t)dxdt
\end{align*}
for all $\tau\in[0,T]$. To handle the right-hand side, we recast the
arguments leading to \eqref{4.3} and end up with the inequality
\[
\int_{\Omega}v_{k}^{2}(x,\tau)dx+d_{0}\int_{0}^{\tau}\int_{\Omega}|\nabla
v_{k}|^{2}dxdt\leq
\|a\|_{L^{\infty}(\Omega_{\infty})}\int_{0}^{\tau}\|v_{k}(\cdot,t)\||A_{k}(t)|^{1/2}dt.
\]
Exploiting the Gr\"{o}nwall lemma arrives at the bound
\[
\|v_{k}\|_{L^{\infty}(0,\tau;L^{2}(\Omega))}+\|\nabla
v_{k}\|_{L^{2}(\Omega_{\tau})}\leq C_{10}(f_{\tau}(k))^{1/2}
\]
with $C_{10}=\sqrt{2\|a\|_{L^{\infty}(\Omega_{\infty})}T}$.

\noindent Then repeating the arguments leading to \eqref{3.10*}
entails
\[
f_{\tau}(k_{2})\leq\frac{C_{11}}{(k_{2}-k)^{\frac{2d+4}{d}}}(f_{\tau}(k))^{1+\frac{2}{d}}
\]
for all $k_{2}>k\geq k_{0}$ and each $\tau\in[0,T]$. Here, the
positive quantity $C_{11}$ is independent of $\tau,$ $k_{2}$ and
$k$. In fine, appealing to the classical Stampacchia Lemma ends up
with the equality
\[
f_{\tau}(k_{2})=0
\]
for all $k_{2}\geq
k_{0}+2^{\frac{d+2}{2}}C_{11}^{\frac{d}{2d+4}}(f_{\tau}(k_{0}))^{\frac{1}{d+2}}$.
Collecting this equality with Corollary \ref{c3.1} leads to the
desired estimate
\[
\|u\|_{L^{\infty}(\Omega_{T})}\leq
k_{0}+2^{\frac{d+2}{2}}C_{11}^{\frac{d}{2d+4}}(|\Omega|T)^{\frac{1}{d+2}}
\]
which completes the proof of Theorem \ref{t3}. \qed


\section{Optimal Control}\label{s8}

\noindent In view of the optimal control problem, we focus on a cost
functional that is based on prescribed the target function for the
tumor cell density in $\Omega_{T}$ at the final time $T$ of the
radiation  therapy. It is wroth  recalling that the coefficient
$a(x,t)$ in \eqref{c-1} has the form
\[
a(x,t)=\rho-R(x,t)
\]
with the positive constant $\rho$, and, generally speaking, this
coefficient can be an arbitrary sign.

Accordingly, for assigned $R\in L^{\infty}(\Omega_{T})$, we define
the objective functional
\begin{equation}
J(R) = \int\limits_{0}^{T}\int \limits_{\Omega} { u_{R}(x,t)  \, dx
dt },
\label{eq:objective}
\end{equation}
where $u_{R}=u_{R}(x,t)$ is the (unique) weak solution to
\eqref{c-1}--\eqref{c-3} with given  $R$ and  originated by any
observed initial state $u_{0}$. Clearly, Theorems \ref{Th-ex},
\ref{Th-ex1} and \ref{t3} tell that the objective functional makes
sense for any fixed positive $T$ if $R(x,t)\leq \rho$ in
$\Omega_{T}$, otherwise only for each positive $T\leq T^{*}$.

Then, we introduce the set of admissible controls
\begin{equation}\label{j-0}
\mathcal{U} = \{ R \in L^{\infty}(\Omega_T): 0 \leqslant R(x,t)
\leqslant M \text{ and }\int\limits_{0}^{T}\int
\limits_{\Omega}{R(x,t) \, dx dt } = \Gamma \},
\end{equation}
where  given $M$ and $\Gamma$ are two threshold positive quantities.
It is worth noting that the set $\mathcal{U} $ is well-defined only
if the value $\frac{\Gamma}{M}$ satisfies the bound
\begin{equation}\label{8.1}
0<\frac{\Gamma}{M}\leq T|\Omega|.
\end{equation}

Turning to the optimal control formulation, we aim to search for
$R^{*}=R^{*}(x,t)\in\mathcal{U}$ such that
\begin{equation}
\label{eq:optControl_obj}
J(R^*) = \mathop {\inf} \limits_{R  \in \mathcal{U} } J(R).
\end{equation}

We notice that, if $M\leq\rho$ then the coefficient $a(x,t)$ is
nonnegative and, hence,  Theorem \ref{t3} provides the one-to-one
global weak solvability of \eqref{c-1}-\eqref{c-3} and, accordingly,
one can find $R^{*}$  for any fixed final time $T$. Thanks to
Corollary \ref{c-4}, the same holds if $u_{0}$ satisfies
\eqref{i.2}. In the case either $R(x,t)>\rho$ or \eqref{i.2} does
not hold, the final time is restricted with $T^{*}$ (see Theorem
\ref{Th-ex}).

\textit{Throughout this section, keeping in mind of Theorems
\ref{Th-ex}---\ref{t3} and Corollary \ref{c-4}, we assume that $T$ is any
positive but fixed if $R(x,t)\leq \rho$ or $u_{0}$
satisfies \eqref{i.2}; otherwise $T$ is any
positive but restricted with  $T^{**}$ defined via Theorem
\ref{Th-ex1}.}


\subsection{Existence of the optimal control}\label{s8.1}
Here, we focus on the proof of the following result.

\begin{theorem}\label{lem-e}
Let the requirements of Theorem \ref{Th-ex1} hold.  Then there
exists an optimal control $R^* \in \mathcal{U}$ minimizing the
objective functional $J(R)$.
\end{theorem}
\begin{proof}
To prove this claim we will exploit the following technique
containing two principal stages. Since  the existence of the optimal
control $R^{*}$ means the existence of the minimum of the objective
functional, the necessary condition providing it is the boundedness
of $J(R)$. Thus, the first stage in our arguments is the
verification of  the estimate
\begin{equation}\label{j-1}
0 \leqslant J(R) \leqslant \frac{(e^{(\rho+M) (1+C_0) \,T }
-1)|\Omega| \|u_0\|_{L^{\infty}(\Omega)} }{(\rho+M)(1+C_0) }
\end{equation}
for each $R\in\mathcal{U}$ and any fixed positive $T$.

 As for the second step, taking into account this bound, we can
conclude that the minimum of $J(R)$ attains  via  the limit of a
minimizing sequence $\{R_{n}\}_{n\in\mathbb{N}},$
$R_{n}\in\mathcal{U},$ i.e.
\begin{equation*}\label{j-2}
\underset{n\to+\infty}{\lim}J(R_{n})=\underset{R\in\mathcal{U}}{\inf}J(R).
\end{equation*}
Then, employing $R_{n}$, we construct  the corresponding sequence of
 nonnegative weak solutions $u_{n}=u_{R_{n}}(x,t)$ to
\eqref{c-1}--\eqref{c-3} with $a_{n}(x,t)=\rho-R_{n}(x,t)$. After
that, passing to the limit in $u_{n}$ and noting  this limit as
$u^{*}=u^{*}(x,t)$, we aim to obtain the following relations
\begin{equation}\label{8.4}
\underset{R\in\mathcal{U}}{\inf}J(R)=\underset{n\to+\infty}{\lim}\int_{\Omega_{T}}u_{n}(x,t)dxdt=\int_{\Omega_{T}}u^{*}(x,t)dxdt=J(R^{*}),
\end{equation}
which in turn complete the proof of Theorem \ref{lem-e}.

\noindent\textit{Step I: Verification of \eqref{j-1}.} Working
within the assumptions of Theorem \ref{lem-e}, we can utilize either
Theorem \ref{t3} or \ref{Th-ex} (depending on the sign of $a(x,t)$),
which immediately provides the one-valued
 weak solvability of \eqref{c-1}--\eqref{c-3} for each
$R\in\mathcal{U}$. Besides, $u_{R}=u_{R}(x,t)$ is nonnegative and
bounded, i.e.
\[
\|u_{R}\|_{L^{\infty}(\Omega_{T})}\leq C_{0}.
\]
Obviously, the nonnegativity of $u_{R}$  entails the nonnegativity
of $J(R)$.

Coming to the proof of the right inequality in \eqref{j-1}, it is a
simple consequence of the following bound to the function $u_{R}$.
\begin{equation}\label{8.2}
\int \limits_{\Omega } { u_{R}(x,\tau)  \, dx  } \leqslant e^{\|a
\|_{L^\infty(\Omega_{T})} (1+C_0) \,\tau } \int \limits_{\Omega } {
u_0(x)  \, dx  }
\end{equation}
for each $\tau \in [0,T].$

To verify this bound, we multiply \eqref{c-1} by $u_{R}$ and, after
integrating over $\Omega_{\tau}$ and performing the standard
calculations, we deduce
\begin{align*}
\int_{\Omega}u_{R}(x,\tau)dx&\leq
\int_{\Omega}u_{0}(x)dx+\int_{\Omega_{\tau}}a(x,t)u_{R}(1-u_{R})dxdt\\&\leq
\int_{\Omega}u_{0}(x)dx+\|a\|_{L^{\infty}(\Omega_{T})}(1+C_{0})\int_{\Omega_{\tau}}u_{R}(x,t)dxdt\\&
\leq
\int_{\Omega}u_{0}(x)dx+[\rho+M](1+C_{0})\int_{\Omega_{\tau}}u_{R}(x,t)dxdt
\end{align*}
with any $\tau\in[0,T]$. Here, we utilized the nonnegativity and the
boundedness of $u_{R}$ and the easily verified estimate
 \[
\|a\|_{L^{\infty}(\Omega_{T})}\leq M+\rho
\]
for each $R\in\mathcal{U}$.

 In fine, applying  Gr\"{o}nwall Lemma \cite{Gr} ends up
with the searched bound, which completes the proof of \eqref{8.2}
and, accordingly, of the boundedness of $J(R)$.

\noindent\textit{Step II: The Proof of \eqref{8.4}.} Coming to the
sequences $\{R_{n}\}_{n\in\mathbb{N}}$ and
$\{u_{n}\}_{n\in\mathbb{N}}$, we apply \eqref{j-1} and Theorem~\ref{Th-ex1} or \ref{t3} and get (after passing to a subsequence):

\noindent$\bullet$ there exists $R^{*}\in L^{\infty}(\Omega_{T})$
satisfying the relations
\begin{equation}\label{j-8}
R_{n}\underset{n\to+\infty}{\rightarrow}R^{*}\quad \text{weakly-*
in}\quad L^{\infty}(\Omega_{T})\quad\text{and} \quad
\int_{\Omega_{T}}R^{*}dxdt=\Gamma;
\end{equation}
\noindent$\bullet$ $u_{n}$ is the unique nonnegative weak solution
of \eqref{c-1}--\eqref{c-3} with $R_{n}$ in place $R$ in the
coefficient $a(x,t)$, and
\begin{equation}\label{j-3}
\|u_{n}\|_{V^{2}(\Omega_{T})}+\|\partial_{t}u_{n}\|_{L^{2}(0,T;(H^{1}(\Omega))^{*})}+
\|u_{n}\|_{L^{\infty}(\Omega_{T})}\leq C^{**}
\end{equation}
with positive $C^{**}$ depending only on
$\|u_{0}\|_{L^{\infty}(\Omega)},$ the Lebesgue measure of $\Omega$,
and the parameters of $d,d_{0},\rho$ and $M$.

Moreover, inequality \eqref{j-3} provides the existence of $u^{*}\in
V_{2}(\Omega_{T})\cap L^{\infty}(\Omega_{T})$ such that
\begin{equation}\label{j-4}
\begin{cases}
u_{n}\underset{n\to+\infty}{\rightarrow}
u^{*}\qquad\quad\text{weakly in}
\quad L^{2}(0,T;H^{1}(\Omega)),\\
u_{n}\underset{n\to+\infty}{\rightarrow}
u^{*}\qquad\quad\text{weakly-* in}
\quad L^{\infty}(0,T;L^{2}(\Omega)),\\
\partial_{t}u_{n}\underset{n\to+\infty}{\rightarrow}
\partial_{t}u^{*}\quad\text{ weakly in}
\quad L^{2}(0,T; (H^{1}(\Omega))^{*}),\\
u_{n}\underset{n\to+\infty}{\rightarrow} u^{*}\qquad\quad\text{
strongly in} \quad L^{2}(\Omega_{T})\quad\text{and a.e. in}\quad
\Omega_{T},
\end{cases}
\end{equation}
and hence, the bound \eqref{j-3} holds with $u^{*}$ in place
$u_{n}$.

 In order to prove that $R^{*}$ solves the optimal control
problem, we are left to verify that

\noindent(i) $R^{*}\in\mathcal{U};$

\noindent(ii) $u^{*}$ is the unique weak solution of
\eqref{c-1}--\eqref{c-3} with $R^{*}$ in place $R$.

\noindent To examine (i), bearing in mind \eqref{j-8} and the
definition of $\mathcal{U}$, we just need to show that
\[
\text{meas } \mathfrak{G}=\text{meas }\{(x,t)\in\Omega_{T}:\quad
R^{*}(x,t)>M \}=0.
\]
Here, arguing by contradiction, that is assuming  $\text{meas }
\mathfrak{G}>0$ and keeping in mind $R_{n}\in\mathcal{U}$, we arrive
at
\[
0\leq
\underset{n\to+\infty}{\lim}\int_{\mathfrak{G}}(M-R_{n}(x,t))dxdt=\int_{\mathfrak{G}}(M-R^{*}(x,t))dxdt<0.
\]
This contradiction can be fixed only if $|\mathfrak{G}|=0$, which
means
\[
R^{*}(x,t)\leq M\quad\text{a.e. in}\quad\Omega_{T}.
\]
The similar arguments provides the nonnegativity $R^{*}$ a.e. in
$\Omega_{T}$. Summing up, we end up with $R^{*}\in\mathcal{U}$.

\noindent$\bullet$ Coming to the verification of the statement in
(ii) and taking into account \eqref{j-3} and \eqref{j-4}, we need
only to examine that $u^{*}$ satisfies \eqref{ident}.  For each
$n\in\mathbb{N}$, $u_{n}$ solves  (in a weak sense)
\eqref{c-1}--\eqref{c-3} with $R_{n}$ in place $R$. In particular,
that means for any $\phi\in L^{2}(0,T;H^{1}(\Omega))$ and each
$n\in\mathbb{N}$ the identity holds
\[
\int \limits_{0}^T \langle \partial_t u_n, \phi \rangle \, dt  +
\int \limits_{\Omega_T}  D (x)  \nabla u_n   \nabla \phi \,dx dt =
\int \limits_{\Omega_T}  (\rho - R_n(x,t) )  u_n  (1 - u_n ) \phi
\,dx dt .
\]
The weak convergence established in \eqref{j-4} provides that
$$
\int \limits_{\Omega_T}  D (x)  \nabla u_n   \nabla \phi \,dx dt
\mathop{\rightarrow} \limits_{n \to +\infty} \int \limits_{\Omega_T}
 D (x)  \nabla u^*   \nabla \phi \,dx dt ,
$$
$$
\int \limits_{0}^T  \langle \partial_t u_n, \phi \rangle  \, dt
\mathop{\rightarrow} \limits_{n \to +\infty} \int \limits_{0}^T
\langle \partial_t u^*, \phi \rangle\, dt .
$$
Now, we are left to justify the corresponding convergence in the
nonlinear term. Indeed, it is a simple consequence of the following
equalities:
\begin{equation}\label{8.5}
\underset{n\to+\infty}{\lim} l_{n}^{1}=0\quad\text{and}\quad
\underset{n\to+\infty}{\lim} l_{n}^{2}=0,
\end{equation}
where we set
\begin{align*}
l_{n}^{1}&=\int_{\Omega_{T}}(R-R_{n})u^{*}(1-u^{*})\phi dxdt,\\
l_{n}^{2}&=\int_{\Omega_{T}}R_{n}[u_{n}(1-u_{n})-u^{*}(1-u^{*})]\phi
dxdt.
\end{align*}
Obviously, the first equality in \eqref{8.5} follows immediately
from the weak-* convergence of  $R_{n}$ (see \eqref{j-8}) and the
belonging $u^{*}(1-u^{*})\phi$ to $L^{1}(\Omega_{T})$, the latter is
provided with the following bound
\[
\int_{\Omega_{T}}|u^{*}||1-u^{*}||\phi|dxdt\leq
C^{**}(1+C^{**})\|\phi\|_{L^{2}(\Omega_{T})}\sqrt{|\Omega|T}.
\]
As for the second equality in \eqref{8.5},  keeping in mind
$R_{n}\in\mathcal{U}$ and \eqref{j-3}, we have
\begin{align*}
|l_{n}^{2}|&\leq M\|u_{n}-u^{*}\|_{L^{2}(\Omega_{T})}\|\phi
\|_{L^{2}(\Omega_{T})}+M\int_{\Omega_{T}}|u^{*2}-u_{n}^{2}||\phi|dxdt\\
& \leq M\|u_{n}-u^{*}\|_{L^{2}(\Omega_{T})}\|\phi
\|_{L^{2}(\Omega_{T})}[1+2C^{**}|\Omega|T]\\
&
\equiv C \|u_{n}-u^{*}\|_{L^{2}(\Omega_{T})}\|\phi
\|_{L^{2}(\Omega_{T})}.
\end{align*}
Here, we again used estimate \eqref{j-3}.

Finally, collecting this estimate with the strong convergence of
$u_{n}$ in $L^{2}(\Omega_{T})$ ends ups with the desired equality.
Thus, we claim that $u^{*}$ satisfies \eqref{ident}. We notice that
the uniqueness of $u^{*}$ is verified with the arguments of Section
\ref{s4.2*}. This completes the proof of Theorem \ref{lem-e}.
\end{proof}


\subsection{Properties of the optimal control: the sensitivity, the necessary condition and the uniqueness}\label{s8.2}
Once the existence of an optimal control is claimed, the next aim is
devising a necessary condition for a control to be optimal. To this
end, we first focus on the differentiability of the mapping $R
\mapsto u_{R}$ with $R\in\mathcal{U}$ and $u_{R}$ being the unique
weak solution of \eqref{c-1}--\eqref{c-3}. To define this
derivative, so-called  the \emph{sensitivity}, we introduce a
function $\Psi_{\varepsilon}$ for each positive $\varepsilon$ and
every $R\in\mathcal{U}$ as
\[
\Psi_{\varepsilon}=\frac{u_{R_{\varepsilon}}-u_{R}}{\varepsilon},
\]
where
\[
R_{\varepsilon}=R+\varepsilon\mathcal{R}
\]
with $\mathcal{R}\in L^{\infty}(\Omega_{T})$, and
$u_{R_{\varepsilon}}$ is a weak solution of \eqref{c-1}--\eqref{c-3}
with $R_{\varepsilon}$ in place $R$.

Due to the boundedness of $R$ and $\mathcal{R}$, there holds
\[
\underset{\varepsilon\to 0}{\lim}\,R_{\varepsilon}=R.
\]

\begin{proposition}\label{p8.1}
Let assumptions of Theorem \ref{Th-ex1} hold. Then
$\Psi_{\varepsilon}\in V_{2}(\Omega_{T})\cap L^{\infty}(\Omega_{T})$
for each positive $\varepsilon$, in particular,
\[
\|\Psi_{\varepsilon}\|_{V^{2}(\Omega_{T})}+\|\Psi_{\varepsilon}\|_{L^{\infty}(\Omega_{T})}+\|\partial_{t}\Psi_{\varepsilon}\|_{L^{2}(0,T;(H^{1}(\Omega))^{*})}\leq
C,
\]
where the positive $C$ is independent of
$\varepsilon\in(0,\varepsilon_{0}]$ for any fixed positive
$\varepsilon_{0}$. Besides,
\begin{equation*}\label{jjj}
J(R_{\varepsilon}) = J(R) + \varepsilon \int \limits_{\Omega_{T}}
\Psi_{\varepsilon}(x,t) \,dx dt .
\end{equation*}
\end{proposition}
\begin{proof}
Obviously, the last equality in this proposition is a simple
consequence of the straightforward calculations and properties of
the function $\Psi_{\varepsilon}$ established in this claim. Hence,
we are left to verify the regularity of $\Psi_{\varepsilon}$, i.e.
$\Psi_{\varepsilon}\in V_{2}(\Omega_{T})\cap
L^{\infty}(\Omega_{T})$.

Recalling that $u_{R}$ is a (unique) weak solution of
\eqref{c-1}--\eqref{c-3} while $u_{\varepsilon}=u_{R_{\varepsilon}}$
solves (in a weak sense) the problem
\begin{equation*}\label{j-11}
\begin{cases}
\partial_{t}u_{\varepsilon} - \nabla (D(x) \nabla u_\varepsilon ) = (a(x,t) - \varepsilon \mathcal{R} )
u_{\varepsilon}(1-u_{\varepsilon}) \qquad   \text{ in }\quad \Omega_T, \\
\nabla u_\varepsilon \cdot \textbf{n}  = 0  \qquad\quad\text{on}\quad \partial \Omega_T, \\
 u_\varepsilon (x,0) = u_0\qquad \text{ in }\quad \Omega.
\end{cases}
\end{equation*}
Clearly,
\[
\|a-\varepsilon\mathcal{R}\|_{L^{\infty}(\Omega_{T})}\leq
\|a\|_{L^{\infty}(\Omega_{T})}+\varepsilon\|\mathcal{R}\|_{L^{\infty}(\Omega_{T})}\leq
\rho+
\|R\|_{L^{\infty}(\Omega_{T})}+\varepsilon_{0}\|\mathcal{R}\|_{L^{\infty}(\Omega_{T})}.
\]
Then working within the assumptions of Theorem \ref{Th-ex1}, we
conclude that
\begin{equation*}\label{8.6}
\|u_{\varepsilon}\|_{V^{2}(\Omega_{T})}+\|u_{\varepsilon}\|_{L^{\infty}(\Omega_{T})}+\|\partial_{t}u_{\varepsilon}\|_{L^{2}(0,T;(H^{1}(\Omega))^{*})}\leq
C
\end{equation*}
with the positive quantity $C$ being independent of
$\varepsilon\in(0,\varepsilon_{0}]$ and, besides, the same estimate
holds for $u_{R}$ in place $u_{\varepsilon}$.

Collecting these facts with the representation of
$\Psi_{\varepsilon}$, we arrive at the initial-boundary value
problem for the linear parabolic equation with bounded coefficients
and the right-hand side being independent of
$\varepsilon\in(0,\varepsilon_{0}]$:
\begin{equation*}\label{j-12}
\begin{cases} \partial_{t}\Psi_{\varepsilon} - \nabla (D(x) \nabla
\Psi_\varepsilon ) - a(x,t)(1-u_{R}-u_\varepsilon)\Psi_\varepsilon =
- \mathcal{R}
u_\varepsilon (1-u_\varepsilon)\quad \text{ in } \Omega_T, \\
\nabla \Psi_\epsilon \cdot \textbf{n}  = 0 \qquad\quad \text{ on } \partial \Omega_{T}, \\
\Psi_\epsilon (x,0) = 0 \qquad\quad  \text{ in } \Omega.
\end{cases}
\end{equation*}
In fine, exploiting the standard theory of linear parabolic
equations to this problem (see, e.g. \cite[Chapter 3]{LSU}), we end
up with the one-valued weak solvability satisfying the desired
bound. This completes the proof of Proposition \ref{p8.1}.
\end{proof}
\begin{definition}\label{d2}
For each $R\in\mathcal{U},$ the derivative of the mapping $R\mapsto
u_{R}$,
 denoted by $\Psi,$ is called the weakly limit of
$\Psi_{\varepsilon}$ (with $R_{\varepsilon}\in\mathcal{U}$) in
$L^{2}(0,T;H^{1}(\Omega))$ as $\varepsilon$ tends to zero.
\end{definition}
Collecting arguments of Proposition \ref{p8.1} with the proof of
Theorem \ref{Th-ex} provides immediately the following claim.
\begin{lemma}\label{lem-s}
Under assumptions of Theorem \ref{Th-ex1},  there is the
\textit{sensitivity} $\Psi$ of
  the mapping $R \mapsto u_R $ with each $R\in\mathcal{U}$, which solves (in a weak sense) the following
problem
\begin{equation}\label{j-10}
\begin{cases}
\Psi_{t} - \nabla (D(x) \nabla \Psi ) - a(x,t)(1-2u_{R}) \Psi = - \mathcal{R} u_{R}(1-u_{R}) \quad \text{ in } \Omega_T, \\
\nabla \Psi \cdot \textbf{n}  = 0  \qquad\text{ on } \partial \Omega_{T}, \\
\Psi(x,0) = 0 \qquad\text{ in } \Omega.
\end{cases}
\end{equation}
Moreover,
\[
0\leq \underset{\varepsilon\to
0}{\lim}\frac{J(R_{\varepsilon})-J(R)}{\varepsilon}\bigg|_{R=R^{*}}=\int_{\Omega_{T}}\Psi(x,t)dxdt
\]
for any $R_{\varepsilon}\in \mathcal{U}$.
\end{lemma}

Now, we characterize our optimal control solution $R^*$ by
differentiating the map $R \mapsto J(R)$. We use the sensitivity
equation to find the so-called "adjoint" problem.  Indeed,
performing the formal ``adjoint''  analysis  in the sensitivity
equation (\ref{j-10}) at $R=R^*$ and exploiting Lemma \ref{lem-s},
we end up with the "adjoint" problem to \eqref{c-1}--\eqref{c-3} for
given optimal control $R^{*}$ and corresponding state $u^{*}$,
\begin{equation}\label{j-13}
\begin{cases}
\Phi_{t} + \nabla (D(x) \nabla \Phi ) + (\rho - R^*)(1-2u^*) \Phi = - 1 \quad \text{ in } \Omega_T, \\
\nabla \Phi \cdot \textbf{n}  = 0  \quad\text{ on } \partial \Omega_T, \\
\Phi(x,T) = 0 \quad\text{ in } \Omega.
\end{cases}
\end{equation}
It is worth noting that the inhomogeneous term, $-1$, in the
equation comes from the integrand of the objective functional with
the respect to state; while the final time condition on the adjoint
function is the transversality condition.

Introducing the new unknown function $\Theta(x,t)=\Phi(x,T-t)$ , we
reduce  \eqref{j-13} to the linear initial-boundary value problem to
a parabolic equation with the  coefficient $b(x,t) = (\rho -
R^*(x,T-t))(1-2u^*(x,T-t)),$
\begin{equation}\label{j-14}
\begin{cases}
 \Theta_{t} - \nabla (D(x) \nabla \Theta) - b(x,t) \Theta  =  1 \quad \text{ in } \Omega_T, \\
\nabla \Theta \cdot \textbf{n}  = 0 \quad \text{ on } \partial \Omega_T, \\
\Theta(x,0) = 0 \quad \text{ in } \Omega,
\end{cases}
\end{equation}
It is worth noting that   Theorem \ref{Th-ex1} provides  the
boundedness in $\Omega_{T}$ of the coefficient $b(x,t)$.

After that, applying the arguments leading to Theorems \ref{Th-ex}
and \ref{Th-ex1} (with exploiting the classical Gr\"{o}nwall lemma
in place the nonlinear Gr\"{o}nwall inequality) to problem
\eqref{j-14} and then returning to the function $\Phi$, we claim the
following result to solution of the problem \eqref{j-13}.
\begin{lemma}\label{Th-opt}
For given  optimal control $R^*$ and the corresponding state $u^*$,
there exists a unique weak non-negative solution $\Phi \in
V^2(\Omega_{T})$ with  $
\partial_t \Phi \in L^2(0, T; (H^1(\Omega))^*)$ and
\[
\|\Phi\|_{ L^{\infty}(\Omega_{T})}\leq C_{\Phi},
\]
where $C_{\phi}$ depends only on the given parameters and the
corresponding norms of the given functions in
\eqref{c-1}--\eqref{c-3}.
\end{lemma}

Now, we are ready to describe the necessary condition of the
optimality.
\begin{theorem}\label{Th-nec}
Let $R^*$ be an optimal control and $u^*$ be the corresponding state. Then
 the following  necessary  condition of the optimality holds
\begin{equation}\label{j-nc}
\int\limits_{0}^{T}\int \limits_{\Omega} { R^*(x,t) g(x,t) \, dx dt}
\leqslant\int\limits_{0}^{T}\int \limits_{\Omega} { R(x,t) g(x,t)
\, dx dt} \ \ \forall\,R \in \mathcal{U},
\end{equation}
where $g(x,t)= \Phi \, u^*( u^* -1) $.
\end{theorem}
\begin{proof}
Assume that $R^*$ is an optimal control and $\mathcal{R} \in
L^{\infty}(\Omega_T)$ such that $R^* + \varepsilon \mathcal{R} \in
\mathcal{U}$ for some $\varepsilon > 0$. Let $ u_{R^* + \varepsilon
\mathcal{R}}$ be the unique solution of (\ref{c-1})--(\ref{c-3}))
when the control term is $R^* + \varepsilon \mathcal{R}$. Using the
formulation of a weak solution  for the "adjoint" problem (\ref{j-13})
with the test function $\Psi$ and keeping in mind Lemma \ref{lem-s},
we arrive at
\begin{align*}
0 &\leq \int \limits_{\Omega_T} {\Psi  \, dx dt}  = \int
\limits_{\Omega_T}  \Phi \Psi_t  \, dx dt -
\int \limits_{\Omega_T}  \Phi \bigg(  \nabla (D(x) \nabla \Psi)+ (\rho - R^*)(1-2u^*) \Psi\bigg)   \, dx dt\\
&= \int \limits_{\Omega_T}  \Phi \mathcal{R} u^*(u^*-1)  \, dx dt ,
\end{align*}
which means
\begin{equation*}\label{j-15}
0 \leqslant \int \limits_{\Omega_T} \Psi  \, dx dt =   \int
\limits_{\Omega_T}  \Phi \mathcal{R} u^*( u^* -1)  \, dx dt.
\end{equation*}
In fine, we consider an arbitrary function $R \in \mathcal{U}$ and
set $\mathcal{R} = R - R^*$. It is easy to check that the set
$\mathcal{U}$ is convex. Thanks to the last relations, we end up
with the inequality
$$
\int \limits_{\Omega_T}  (R - R^* )\Phi \, u^*( u^* -1)  \, dx dt
\geqslant 0 \ \ \forall\,R \in \mathcal{U},
$$
which in turn  arrives at the  necessary optimality condition
(\ref{j-nc}).
\end{proof}

Our last result concerns with the representation of $R^{*}$ and with
the uniqueness of the optimal control. Before formulating this
claim, we describe some important properties of the function $g$
given in Theorem \ref{Th-nec}.
\begin{proposition}\label{p8.2}
Let \eqref{2.2} hold and
$$
0 \leqslant u_0(x) \leqslant 1,  \ \ \int \limits_{\Omega}{(| \ln(
u_0)| + | \ln(1-u_0)| ) \,dx} < \infty.
$$
We assume that  $a(x,t)=\rho-R^{*}$ and $g=\Phi u^{*}(u^{*}-1)$,
where $R^*$ and $u^*$ are the optimal control and
 its corresponding state, $\Phi$ is a solution of the adjoint problem
 \eqref{j-13}. Then the function $g$ satisfies the inequalities:
 \[
\mes\{(x,t)\in\Omega_{T}:\quad g=0\}=0\qquad\text{and}\quad
-\frac{C_{\phi}}{4}\leq g<0\quad\text{a.e. in}\quad\Omega_{T}.
 \]
\end{proposition}
\begin{proof}
It is apparent that the first equality is a simple consequence of
Corollary \ref{R-1} and Lemma~\ref{Th-opt}, which (in particulary)
provide that
\begin{equation}\label{9.0}
u^{*},\Phi\neq 0\quad\text{and}\quad u^{*}\neq 1 \quad\text{ a.e.
in}\quad \Omega_{T}.
\end{equation}

As for the remaining inequalities in this claim, they  follow  from
Lemma~\ref{Th-opt} and Theorem \ref{Th-ex1} or \ref{Th-2} (depending
on the sign of $a(x,t)$).
 Indeed,  utilizing \eqref{9.0} with Lemma~\ref{Th-opt} and Theorem \ref{Th-ex1} or
 \ref{Th-2}
yields
\begin{equation*}\label{9.1}
0< \Phi\leq C_{\Phi}\qquad\text{and}\qquad 0< u^{*}<
1\quad\text{a.e.}\quad\text{in}\quad\Omega_{T}.
\end{equation*}
Finally,  exploiting the inequalities related with $u^{*}$  and
performing technical calculations, we conclude that
\begin{equation*}\label{9.2}
0< u^{*}(1-u^{*})\leq \frac{1}{4}\quad\text{a.e. in}\quad\Omega_{T}.
\end{equation*}
In fine, collecting all obtained estimates ends up with the desired
bounds, which finishe the proof of this proposition.
\end{proof}
Denoting
\[
\mathfrak{W}=\frac{\Gamma}{M}
\]
for each given threshold values $\Gamma$ and $M$ satisfying
\eqref{8.1}, we define the quantity
\[
\kappa^{*}=\sup\Big\{\kappa\in\Big[-\frac{ C_{\Phi}}{4},0\Big):\,
\mes \{(x,t)\in\Omega_{T}:\, g(x,t)<\kappa\}\leq \mathfrak{W}\Big\}
\]
and introduce the sets
\[
\mathcal{E}_{1}=\{(x,t)\in\Omega_{T}:\,
g(x,t)<\kappa^{*}\}\qquad\text{and}\qquad
\mathcal{E}_{2}=\{(x,t)\in\Omega_{T}:\, g(x,t)=\kappa^{*}\}.
\]
\begin{theorem}\label{Th-un}
Let \eqref{2.2} hold and
$$
0 \leqslant u_0(x) \leqslant 1,  \ \ \int \limits_{\Omega}{(| \ln(
u_0)| + | \ln(1-u_0)| ) \,dx} < \infty.
$$
We assume that $R^*$ and $u^*$ are the optimal control and
 its corresponding state.
Then the optimal control is unique if either $ \mathfrak{W}=
|\Omega|T$ or $\mathfrak{W}<|\Omega|T$ and $ |\mathcal{E}_{2}|=0$.
Besides,
\begin{equation*}\label{10.1}
R^{*}=\begin{cases} M\chi_{\Omega_{T}}\qquad\text{if}\quad \mathfrak{W}= |\Omega|T,\\
M\chi_{\mathcal{E}_{1}} \quad\quad\text{ if}\quad
\mathfrak{W}<|\Omega|T\quad\text{and}\quad |\mathcal{E}_{2}|=0.
\end{cases}
\end{equation*}
If $\mathfrak{W}<|\Omega|T$ and the set $\mathcal{E}_{2}$ has a
positive Lebesgue measure, then $R^{*}$ is written in the form
\[
R^{*}=M\chi_{\mathcal{E}_{1}}+CM\chi_{\mathcal{E}_{2}}
\]
with  constant
$C=\frac{\mathfrak{W}-|\mathcal{E}_{1}|}{|\mathcal{E}_{2}|}$.
 This minimizer is unique if either
$\mathfrak{W}=|\mathcal{E}_{1}|$ or
$\mathfrak{W}=\{(x,t)\in\Omega_{T}:\, g(x,t)\leq\kappa^{*}\}$.
\end{theorem}
\begin{proof}
First, we prove this claim in the special case $M=1$, and then we
describe how this restriction can be removed. Theorem \ref{Th-nec}
dictates us that a minimizer of $J(R)$ should be the bang-bang type.
In particular, it means that $R^{*}$ is also a minimizer of the
problem
\begin{equation}\label{11.1}
\hat{J}(R)=\underset{R\in\mathcal{U}}{\inf}\int_{\Omega_{T}}R(x,t)g(x,t)dxdt.
\end{equation}
Applying Bathtub principle (see, e.g., \cite[Theorem 1.14, Chapter
1]{LL}) to problem \eqref{11.1} provides a solution $R^{*}$ in the
form
\begin{equation}\label{11.2}
R^{*}=\chi_{\widehat{\mathcal{E}}_{1}}+C\chi_{\widehat{\mathcal{E}}_{2}},
\end{equation}
where
\[
\widehat{\mathcal{E}}_{1}=\{(x,t)\in\Omega_{T}:\quad
g<\widehat{\kappa}^{*}\}\quad\text{and}\quad
\widehat{\mathcal{E}}_{2}=\{(x,t)\in\Omega_{T}:\quad
g=\widehat{\kappa}^{*}\}
\]
with the value $\widehat{\kappa}^{*}$ defined as
\[
\widehat{\kappa}^{*}=\sup\{\kappa\in\mathbb{R}:\quad\mes\{(x,t)\in\Omega_{T}:\,
g<\kappa\}\leq\Gamma\},
\]
while the constant $C$ is identified via the equality
$C|\widehat{\mathcal{E}}_{2}|=\Gamma-|\widehat{\mathcal{E}}_{1}|$.

At this point, we utilize Proposition \ref{p8.2} to deduce the
following inequalities:
\begin{align*}
&\mes\{(x,t)\in\Omega_{T}:\, g<\kappa\}=\begin{cases}
0\quad\quad\text{ if}\quad
\kappa\in\Big(-\infty,-\frac{C_{\Phi}}{4}\Big),\\
T|\Omega|\quad\text{if}\quad\kappa\in[0,+\infty),
\end{cases}\\
& 0\leq\mes\{(x,t)\in\Omega_{T}:\, g<\kappa\}\leq
T|\Omega|\quad\text{if}\quad \kappa\in\Big[-\frac{C_{\Phi}}{4},0\Big),\\
&\mes\{(x,t)\in\Omega_{T}:\, g=0\}=0.
\end{align*}
Exploiting these relations and taking into account the restriction
on the $\Gamma$ (see \eqref{8.1} with $M=1$), we find:

\noindent$\bullet$  $ \widehat{\kappa}^{*}= +\infty,$  $
\widehat{\mathcal{E}}_{1}= \Omega_{T}$ and
$|\widehat{\mathcal{E}}_{2}|=0$, if only  $\Gamma=T|\Omega|$;

\noindent$\bullet$
$\widehat{\kappa}^{*}=\kappa^{*}\in\Big[-\frac{C_{\Phi}}{4},0\Big]$,
$ \widehat{\mathcal{E}}_{1}= \mathcal{E}_{1}$ and
$\widehat{\mathcal{E}}_{2}=\mathcal{E}_{2}$, if $\Gamma<T|\Omega|$.

 Collecting these relations with the representation \eqref{11.2}, we
 complete the proof of Theorem \ref{Th-un} in the special case
 $M=1$. In order to verify this claim in the case of arbitrary positive
 $M$,  introducing
 \begin{align*}
\hat{R}&=\frac{R}{M},\quad
\hat{g}=Mg,\quad\hat{\Gamma}=\frac{\Gamma}{M},\\
\widehat{\mathcal{U}}&=\Big\{\hat{R}\in L^{\infty}(\Omega_{T}):\quad
0\leq \hat{R}\leq 1\quad\text{and}\quad
\int_{\Omega_{T}}\hat{R}dxdt=\hat{\Gamma}\Big\},
 \end{align*}
 and performing direct calculations, we reduce  \eqref{11.1}
 to the problem
 \[
\hat{J}(\hat{R})=\underset{\hat{R}\in\widehat{\mathcal{U}}}{\inf}\int_{\Omega_{T}}\hat{R}\hat{g}dxdt.
 \]
After that, arguing similar to the case $M=1$, we arrive at the
desired results. That finishes the proof of Theorem \ref{Th-un}.
\end{proof}


\section{Numerical Study}\label{s9}

\noindent In this section, we numerically simulate the optimal control problem presented in Section \ref{s8}. The nonlinear parabolic equation \eqref{c-1}--\eqref{c-3} is solved forward in time $t\in [0,T]$ and is referred to as the forward problem. The adjoint equation \eqref{j-13} is solved backward in time $\tau = T-t \in [0,T]$ using the formulation \eqref{j-14} and is referred to as the backward problem. To enforce the constraint $\int_{\Omega_T}R(x,t)~dxdt=\Gamma$, we augment the objective function \eqref{eq:objective} to
\begin{equation}
    \tilde{J}(R) = \int \limits_{\Omega_T} { u_{R}(x,t)  \, dx
dt } + \frac{\lambda}{2} \left(\int_{\Omega_T}R\,dxdt - \Gamma\right)^2,
\label{eq:modified_obj}
\end{equation}
where $\lambda$ is a penalty parameter of choice. The optimal control problem \eqref{eq:optControl_obj} is relaxed to
\begin{equation}
\label{eq:optControl_relax}
\tilde{J}(R^*) = \mathop {\inf} \limits_{R  \in \tilde{\mathcal{U}} } \tilde{J}(R), \quad 
\tilde{\mathcal{U}} = \{ R \in L^{\infty}(\Omega_T): 0 \leqslant R(x,t)
\leqslant M \}.
\end{equation}
\begin{algorithm}
\caption{Gradient Descent}
\label{algo:gd}
\vspace{2mm}	
\begin{algorithmic}[1]
    \State \textbf{Input}: $\Gamma, \lambda, \Delta, n, T, M$
    \Comment{$\Delta$:= Step size, $n$:= Number of iterations}
    \State \textbf{Initialize} $u(x,0), \Phi(x,t=T), R_0(x,t)$
    \Comment{$R_0(x,t)$:=Initial guess for the control variable}
    \State Solve for $u(x,t)$ in \eqref{c-1}-\eqref{c-3}, with $R(x,t)$ set to $R_0(x,t)$.
    \State Compute the initial cost $J_0$ \eqref{eq:objective} with $u(x,t)$, $R_0(x,t)$
    \For{$k = 1:n$}
    \State Solve for $\Phi(x,t)$ in \eqref{j-13}, with $R^*$ set to $R$, and $u^*$ set to $u$
    \State Compute $R_c(x,t) = R_{k-1}(x,t) + \Delta \left[\int_{\Omega}\Phi(x,t)u(1-u)~dx - \lambda\left(\int_{\Omega_T} R_{k-1}(x,t)~dxdt - \Gamma\right)\right]$.
    \If{$R_c(x,t_j)<0$ for some $t_j \in [0,T]$, $x \in [0,T]$} 
    \Comment{Projected Gradient}
    \State Set $R_c(x,t_j) = 0$
    \EndIf
    \If{$R_c(x,t_j)>M$ for some $t_j \in [0,T]$, $x \in [0,T]$} 
    \Comment{Projected Gradient}
    \State Set $R_c(x,t_j) = M$
    \EndIf
    \State Solve for $u_c(x,t)$ in \eqref{c-1}-\eqref{c-3}, with $R(x,t)$ set to $R_c(x,t)$.
    \State Compute cost $J_i$ with $u_c(x,t)$, $R_c(x,t)$
    \If{$J_k < J_{k-1}$}
    \State Set $R_{k}(x,t) = R_c(x,t)$, $u(x,t) = u_c(x,t)$, $\Delta = 2\Delta$
    \Else
    \State Set $\Delta = \Delta/2$
    \EndIf
    \EndFor
\end{algorithmic}
\end{algorithm} 
To find the optimal control $R^*(x,t)$, we perform gradient descent on $R$ to numerically solve the optimization problem \eqref{eq:optControl_relax}. Projected gradient descent is used to enforce the constraint $0 \leqslant R\leqslant M$, where $M = 1$ in all simulations.
This algorithm is presented in Algorithm \ref{algo:gd}. 

We numerically simulate the optimal control problem in two settings. In Section \ref{sec:1dControl}, we consider a one-dimensional spatial domain $\Omega = [0, L]$, where the control $u$ is either spatially uniform, $u = u(t)$, or spatially varying, $u = u(x, t)$. In Section \ref{sec:2dControl}, we present the results for optimal control on a two-dimensional domain $\Omega = [0, L] \times [0, L]$. For the one-dimensional case, we use a second-order finite-difference approximation for spatial discretization on a uniform grid, along with backward Euler for time-stepping, for both the forward and backward problems. For the two-dimensional case, we use the finite element package in MATLAB's Partial Differential Equations toolbox to simulate the optimal control problem.

\noindent
\subsection{One-dimensional optimal control}
\label{sec:1dControl}
\begin{figure}
    \centering
    \includegraphics[width=0.6\linewidth]{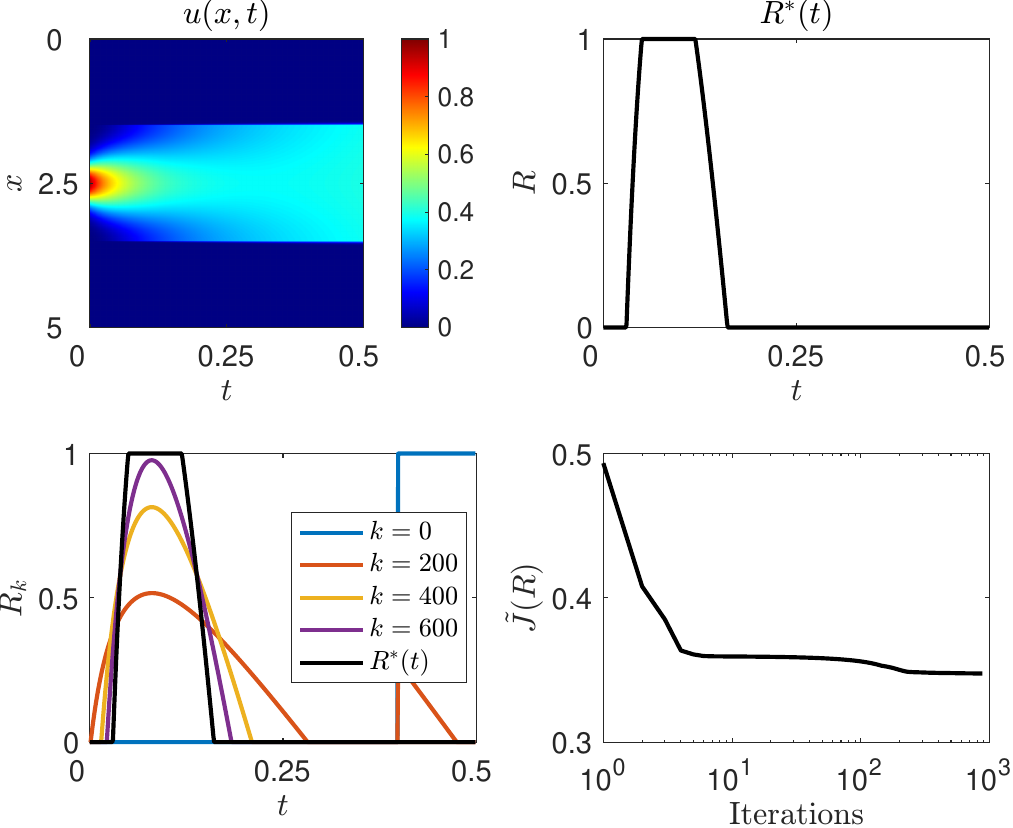}
    \caption{(Top left) controlled $u(x,t)$, (top right) optimal $R(t)$, (bottom left) convergence to the optimal $R(t)$ from an initial guess (blue), and (bottom right) cost functional over iterations.}
    \label{fig:uniform_control}
\end{figure}
We first present numerical results for simulating the optimal control problem over a one-dimensional domain $\Omega=[0,L]$, where $L = 5$. The terminal time is set to be $T = 0.5$. The invasion rates in the white brain tissue, $x \in [0.15,0.35]$, and the grey brain tissue, $x \in [0,0.15)\cup (0.35,0.5]$, are $D_w = 1$ and $D_g = 0.001$, respectively. The initial condition for $u$ is given by
\begin{equation}
u_0(x) = u(x,0) = \exp\left(-8(x-\tfrac{1}{2}L)^2\right).
\label{eq:u_ic}
\end{equation}
\subsubsection{Spatially-uniform control}
For the case of spatially uniform control, $R = R(t)$, 
we set the initial guess for $R(t)$ to be
\begin{equation}
    R_0(t)=
    \begin{cases}
    0 \quad \text{for } 0 \le t < 0.4,\\
    1 \quad \text{for } 0.4 \le t \le 0.5,
    \end{cases}
\label{eq:R_guess_uniform}
\end{equation}
and the corresponding $\Gamma = 0.5$.
The gradient descent algorithm \ref{algo:gd} is applied to solve the optimal control problem, where the penalty coefficient is set to $\lambda = 100$.
Figure \ref{fig:uniform_control} (top) shows the controlled profile $u(x,t)$ and the optimal control $R^{*}(t)$ for the case where $R$ is spatially uniform. Figure \ref{fig:uniform_control} (bottom) illustrates the convergence to $R^{*}(t)$ from the initial guess \eqref{eq:R_guess_uniform} and the history of the objective function $\tilde{J}(R)$ in \eqref{eq:modified_obj} over iterations. The plot of the controlled tumor cell density, $u(x,t)$, shows that the initial Gaussian profile \eqref{eq:u_ic} quickly diffuses into the white brain tissue zone and remains at a low magnitude due to the combined effects of cell proliferation and radiation therapy.
The profile of the desired loss term $R^*(t)$ demonstrates that for this setting, the radiotherapy should only take place for a short period of time at the maximum value $M = 1$, and $R^*$ remains close to zero for the later stage of the time period. 

\subsubsection{Spatially-nonuniform control}
For the case of one-dimensional spatially nonuniform control, $R = R(x,t)$, we set the initial guess for $R(x,t)$ to be spatial and temporal uniform,
\begin{equation}
R_0(x,t) \equiv \Gamma /(LT),
\end{equation}
 where the constraint $\int_{\Omega_T}R_0\,dxdt = \Gamma$ is satisfied.
 The other settings are identical to the ones in the spatially uniform case in Figure \ref{fig:uniform_control}.
Figure \ref{fig:nonuniform_control} shows a comparison of the controlled $u(x,t)$ and the optimal control variable $R^*(x,t)$ for $\Gamma = 0.5$ and $\Gamma = 0.75$, respectively. Similar to the spatially uniform control case, for both values of $\Gamma$, the controlled tumor cell density $u(x,t)$ spreads into the white brain tissue zone. During the early stage, from $t = 0$ to $t = 0.1$, the optimal control $R^*$ reaches its maximum value, $M = 1$, at locations centered around the region of high tumor cell density. As $u$ spreads in space, the positive region of $R^*$ also expands until it reaches the interface between the white/grey brain tissue zones. Similar to the spatially uniform control case, in the non-uniform control case, high-intensity radiotherapy in the optimal control $R^*$ occurs in the early stages and remains low later on, with the duration of the high-intensity period depending on the value of $\Gamma$.


\begin{figure}
     \centering
     \begin{subfigure}{0.6\linewidth}
         \centering
    \includegraphics[width=\linewidth]{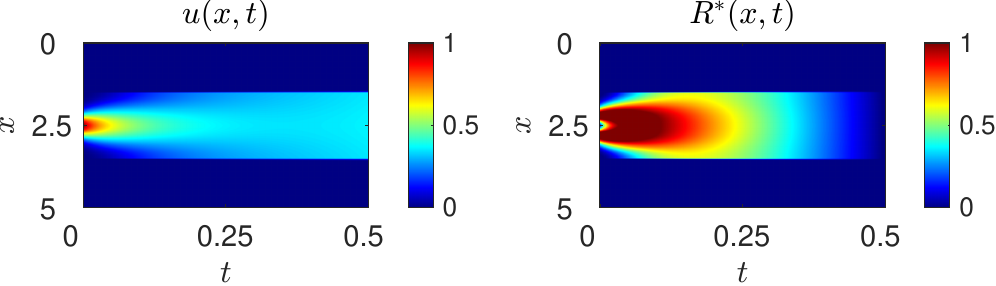}
         \caption{$\Gamma = 0.5$}
         \label{fig:Gamma=5e-1}
     \end{subfigure}
     \hfill
     \begin{subfigure}[b]{0.6\linewidth}
         \centering
    \includegraphics[width=\linewidth]{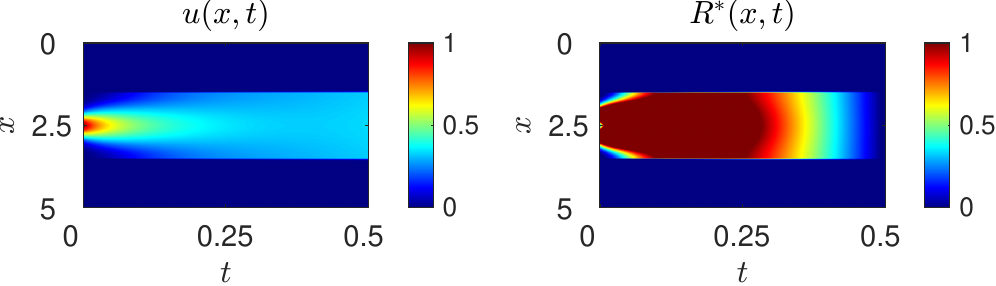}
         \caption{$\Gamma = 0.75$}
         \label{fig:Gamma=75e-2}
     \end{subfigure}
        \caption{Controlled profiles $u(x,t)$ and the control variable $R^*(x,t)$ with (A) the total $\Gamma = 0.5$ and (B) the total $\Gamma = 0.75$.
        }
        \label{fig:nonuniform_control}
\end{figure}

\subsection{Two-dimensional optimal control}
\label{sec:2dControl}
Next, we present the simulation results for the optimal control problem over a two-dimensional space $\Omega = [0,L]\times[0,L]$ with $L = 5$. The terminal time $T$ and invasion rates $D_w$ and $D_g$ are identical to those used in the one dimensional case in Subsection \ref{sec:1dControl}.
The region of white brain tissue is defined as 
\begin{equation}
\Omega_w = \{(x,y)\in \Omega: (x-\tfrac{L}{2})^2+4(y-\tfrac{L}{2})^2 \leqslant 1\},
\label{eq:Omegaw}
\end{equation}
and the region of grey brain tissue $\Omega_g = \Omega \backslash \Omega_w$. We set the initial condition for the tumor cell density to be
\begin{equation}
u_0(x,y) = u(x,y,0) = \exp\left(-5\left[(x-\tfrac{1}{2}L)^2 + (y - \tfrac{1}{2}L)^2\right]\right).
\end{equation}
The initial guess for $R(x,y,t)$ is given by
\begin{equation}
    R_0(x,y,t) = 
    \begin{cases}
    \frac{\Gamma}{|\Omega_w|T},\quad &(x,y)\in \Omega_w,\\
    0, &\mbox{ otherwise}.
    \end{cases}
\end{equation}
where $|\Omega_w|$ is the area of the white brain tissue region \eqref{eq:Omegaw}. This choice ensures that the initial guess $R_0$ satisfies the constraint $\int_{\Omega_T} R_0 = \Gamma$.  We set the total amount $\Gamma = 0.7$ and the penalty coefficient $\lambda = 1$ for this example.

Figure \ref{fig:2d_control} illustrates the evolution of the optimal control variable $R^*(x,y,t)$ and the corresponding controlled profiles $u(x,y,t)$ over time. Starting from a Gaussian profile centered in the domain, the controlled tumor cell density $u$ rapidly diffuses within $\Omega_w$, gradually evolving into a spatially uniform profile at a reduced density. The optimal control $R^*$ between $t = 0$ and $t = 0.3$ forms islands of magnitude $M = 1$, with the shape of these islands transitioning from circular at $t = 0$ to resemble the geometry of $\Omega_w$ by $t = 0.3$. For $t > 0.3$, we observe that the optimal control $R^* \equiv 0$, which is similar to the pattern observed in the one-dimensional cases, where radiotherapy is excluded in the later stages of optimal control.

\begin{figure}
     \centering
     \begin{subfigure}{0.9\linewidth}
         \centering
    \includegraphics[width=1\linewidth]{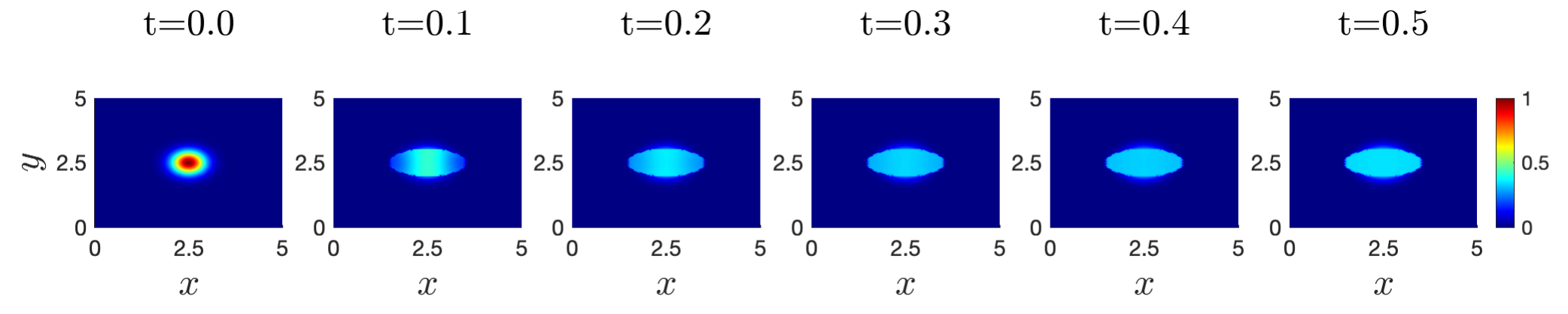}
     \end{subfigure}
     \hfill
     \begin{subfigure}[b]{0.9\linewidth}
         \centering
    \includegraphics[width=1\linewidth]{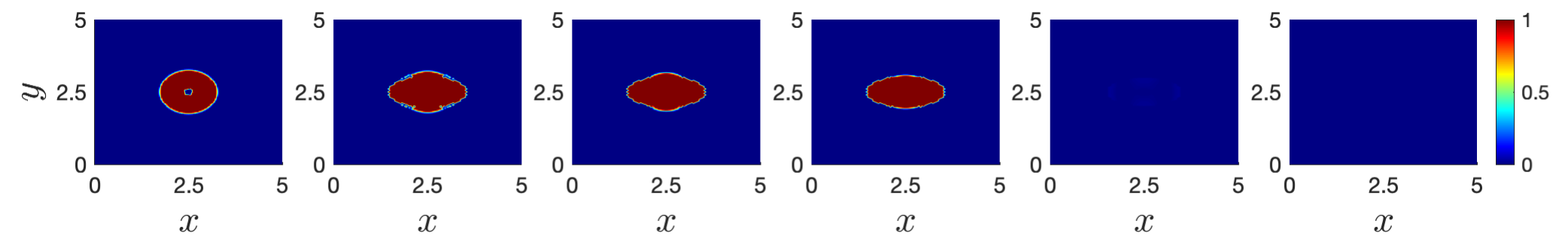}
     \end{subfigure}
        \caption{(Top) Controlled profiles $u(x,y,t)$ and (bottom) the control variable $R^*(x,y,t)$ in time, where $\Gamma = 0.7$.}
        \label{fig:2d_control}
\end{figure}

\section{Conclusions and discussion}

In this paper, we analytically study a parabolic PDE model for tumor cell density influenced by radiation therapy and its associated optimal control problem. We analyze the well-posedness and regularity of the PDE model and investigate an optimal control setting aiming at minimizing the overall tumor density in space and in time, where the loss term representing the effects of radiation therapy, $R(x,t)$, serves as the control. 

Our numerical studies in both one- and two-dimensional cases suggest that for this optimal control problem, the optimal strategy is for radiation therapy to be applied to the region of interest (the white brain tissue) during the early stage of the treatment. During this time, the shape of the high-intensity radiotherapy region resembles that of the high-intensity tumor cell region, until the tumor cells spread to the interface between the white and grey brain tissue. Additionally, within the high-intensity radiotherapy region, the optimal treatment plan maintains the radiation intensity at its maximum value (determined by a set threshold). After this initial high-intensity period, the radiotherapy intensity remains close to zero for the remainder of the treatment.

We expect that the results on the optimal control problem could motivate further work in the direction of optimizing radiotherapy treatment planning. To make the control problem more realistic, we need to incorporate more patient-specific constraints in both temporal and spatial domains for the control variable $R(x,t)$. For example, feasible radiotherapy treatment may only occur during specific time intervals. Additionally, since $R(x,t)$ originates from the effects of external beam radiation, the optimal control may need to satisfy additional spatial constraints. At the modeling level, the governing PDE \eqref{c-1}--\eqref{c-3} assumes that the diffusion coefficient is homogeneous within the grey and white brain tissue, and the intrinsic proliferation rate is constant in space. More sophisticated physiological models and data-driven methodologies \cite{lima2017selection,Jarrett2018mathematical,ji2022post,Yankeelov2013clinically} may be required to improve the applicability of the study.




\begin{thebibliography}{99}

\bibitem{A}
R.A. Adams, Sobolev spaces, Academic Press, New York, 1975.


\bibitem{DiB}
E. Di Benedetto,
 Degenerate parabolic equations,
 Springer-Verlag, New York, 1993.

 \bibitem{kpp1}
H. Berestycki, G. Nadin,  B. Perthame  and L. Ryzhik, The non-local
FisherKPP equation: travelling waves and steady states,
 Nonlinearity, \textbf{81}(3) (2009) 2813.

 \bibitem{BF18}
H. Berestycki  J. Fang,
 Forced waves of the Fisher-KPP equation in a shifting environment,
 J. Differential Equations, \textbf{264} (2018) 2157--2183.

 \bibitem{billy2013}
F. Billy, J. Clairambault  and O. Fercoq,
 Optimisation of cancer drug treatments using cell population
 dynamics,
 Mathematical Methods and Models in Biomedicine, Springer, New York, (2013)  265--309.

\bibitem{bratus2017}
A. Bratus,  I. Samokhin, I. Yegorov and D. Yurchenko,
 Maximization of viability time in a mathematical model of cancer
 therapy,
 Mathematical Biosciences, \textbf{294} (2017) 110--119.

 \bibitem{byun2020}
D.J. Byun,  M.M. Tam, A.S. Jacobson,  M.S. Persky, T.T. Tran, B.
Givi,
 et al,
 Prognostic potential of mid-treatment nodal response in oropharyngeal squamous cell
 carcinoma,
 Head and Neck, \textbf{43}(1) (2021) 173--181. DOI: 10.1002/hed.26467.

\bibitem{CGM}
M. Conti, S. Gatti, A. Miranville, Mathematical analysis of a
phase-field model of brain cancers with chemotherapy and
antiangiogenic therapy effects, AIMS Mathematics, \textbf{7}(1)
(2021) 1536--1561.

\bibitem{DFLLY10}
W. Ding, H. Finotti, S. Lenhart, Y. Lou, Q. Ye,
 Optimal control of growth coefficient on a steady-state population model,
 Nonlinear Analysis: Real World Applications,  \textbf{11}(2) (2010)  688--704.

 \bibitem{Dr}
 S.S. Dragomir, Some Gr\"{o}nwall type inequalities and
 applications, Nova Science Publishers, Melbourne, 2003.

 \bibitem{kpp2}
Y.  Du and Z. Guo,
 The Stefan problem for the FisherKPP equation, J. Differential
 Equations,
\textbf{253}(3) (2012) 996--1035.

 \bibitem{eschrich2009}
S.A. Eschrich,  J. Pramana, H. Zhang,  H. Zhao,  D. Boulware,  J.H.
Lee,
 et al.,
 A gene expression model of intrinsic tumor radiosensitivity: prediction of response and prognosis after
 chemoradiation,
International Journal of Radiation Oncology, Biology, Physics,
\textbf{75}(2) (2009) 489--496. DOI: 10.1016/j.ijrobp.2009.06.014.

 \bibitem{FM77}
P.C. Fife, J.B. McLeod, The approach of solutions of nonlinear
diffusion equations to travelling front solutions, Archive for
Rational Mechanics and Analysis, \textbf{65}(4) (1977) 335--361.

\bibitem{FLV12}
H. Finotti, S. Lenhart, T.V. Phan, Optimal control of advective
direction in reaction-diffusion population models, Evolution Equa.
 Control Theory,  \textbf{1}(1) (2012) 81--107.

\bibitem{Fis37}
R.A. Fisher,  The Wave of Advance of Advantageous Genes,
 Annals of Eugenics, \textbf{7}(4) (1937) 353--369.


\bibitem{gerlee2022}
P. Gerlee, P.M. Altrock, A. Malik,  C. Krona  and S. Nelander,
 Autocrine signaling can explain the emergence of Allee effects in cancer cell
 populations,
 PLOS Computational Biology, \textbf{18}(3) (2022) e1009844. DOI: 10.1371/journal.pcbi.1009844.

 \bibitem{Grin91}
 P. Grindrod, Patterns and waves:
 The theory and applications of reaction-diffusion equations,  Oxford Applied Mathematics and Computing Science
 Series, 1 edt.,  Oxford University Press, Oxford,  1991.

 \bibitem{Gr}
 T.H. Gr\"{o}nwall, Note on the derivatives with respect to a
 parameter of the solution of a system of differential equations,
 Ann. Math., \textbf{20}(2) (1919) 293--296.


\bibitem{hanin2001}
L.G. Hanin,
 Iterated birth and death process as a model of radiation cell
 survival,
 Mathematical Biosciences, \textbf{169}(1) (2001) 89--107. DOI: 10.1016/S0025-5564(01)00068-6.

\bibitem{hanin2004}
L.G. Hanin,
 A stochastic model of tumor response to fractionated radiation: limit theorems and rate of
 convergence,
 Mathematical Biosciences, \textbf{191}(1) (2004) 1--17. DOI: 10.1016/j.mbs.2004.03.001.

\bibitem{hanin2010}
L.G. Hanin, M. Zaider, Cell-survival probability at large doses: an
alternative to the linear-quadratic model, Physics in Medicine and
Biology, \textbf{55}(16) (2010) 4687. DOI:
10.1088/0031-9155/55/15/004.

\bibitem{hanin2013}
L.G. Hanin,  M. Zaider, A mechanistic description of
radiation-induced damage to normal tissue and its healing kinetics,
Physics in Medicine and Biology, \textbf{58}(4) (2013) 825. DOI:
10.1088/0031-9155/58/4/825.

\bibitem{hanin2014}
L.G. Hanin,  M. Zaider,
 Optimal schedules of fractionated radiation therapy by way of the greedy principle: biologically-based adaptive
 boosting,
 Physics in Medicine and Biology, \textbf{59}(15) (2014) 4085. DOI: 10.1088/0031-9155/59/16/4085.

\bibitem{hanin1993}
L. Hanin, S. Rachev  and A.Y. Yakovlev,
 On the optimal control of cancer radiotherapy for non-homogeneous cell
 populations,
 Advances in Applied Probability, \textbf{25}(1) (1993) 1--23. DOI: 10.2307/1427493.

\bibitem{hanin2006}
L.G. Hanin, O. Hyrien,  J. Bedford  and A.Y. Yakovlev, A
comprehensive stochastic model of irradiated cell populations in
culture,
 J. Theoretical Biology, \textbf{239}(4) (2006) 401--416. DOI: 10.1016/j.jtbi.2005.09.016.

 \bibitem{HR16}
F. Hamel, L. Rossi,
 Transition fronts for the Fisher-KPP equation,
Transactions AMS, \textbf{368} (2016) 8675--8713.

 \bibitem{hendry1984}
J.H. Hendry, J.V.  Moore,
 Is the steepness of dose-incidence curves for tumor control or complications due to variation before, or as a result of, irradiation?
British Journal of Radiology, \textbf{57}(683) (1984) 1045--1046.
DOI: 10.1259/0007-1285-57-683-1045.

\bibitem{hlatky1994}
L.R. Hlatky,  P. Hahnfeldt  and R.K. Sachs,  Influence of
time-dependent stochastic heterogeneity on the radiation response of
a cell population,
 Mathematical Biosciences, \textbf{122}(2) (1994) 201--220. DOI: 10.1016/0025-5564(94)90065-5.

\bibitem{hillen2010}
T. Hillen, G.D. Vries, J. Gong  and C. Finlay, From cell population
models to tumor control probability: including cell cycle effects,
Acta Oncologica, \textbf{49} (2010) 1315--1323. DOI:
10.3109/0284186X.2010.485202.

\bibitem{iwasa2006}
Y. Iwasa,  M.A. Nowak  and F. Michor,
 Evolution of resistance during clonal expansion,
 Genetics, \textbf{172} (2006) 2557--2566. DOI: 10.1534/genetics.105.053793.


 \bibitem{KPP37}
A. Kolmogorov, I. Petrovskii, and N. Piskunov,
 A study of the diffusion equation with increase in the amount of substance, and its application to a biological
 problem, In V. M. Tikhomirov, editor, Selected Works of A. N. Kolmogorov
 I,  248-270,
 Kluwer, 1991.

 \bibitem{kutva2023}
 A.R. Kutuva, J.J. Caudell,  K. Yamoah, H. Enderling  and M.U. Zahid,
 Mathematical modeling of radiotherapy: impact of model selection on estimating minimum radiation dose for tumor
 control, Frontiers in Oncology, \textbf{13} (2023) 1--9. DOI: 10.3389/fonc.2023.1130966.

\bibitem{LL} E.H. Lieb, M. Loss,
 Analysis, Graduate Studies in
Mathematics, v. 14, Providence, RI, 2001.


\bibitem{LSU}
O.A. Ladyzhenskaia, V.A. Solonnikov, N.N. Uraltseva,
 Linear and quasilinear parabolic equations,
 Academic Press,  New York, 1968.

\bibitem{LM}
J.-L. Lions, E. Magenes,
 Probl{\`e}mes aux limites non homog{\`e}nes et applications, v 1,
 Monographies universitaires de math{\'e}matiques, Dunod, 1968.


\bibitem{meaney2019}
C. Meaney,  M. Stastna, M. Kardar and M. Kohandel, Spatial
optimization for radiation therapy of brain tumours, PLOS ONE,
\textbf{14} (2019) e0217354.

\bibitem{Mur02}
 J.D. Murray,
 Mathematical Biology, I: An Introduction, 3 ed., Springer, New York, 2002.

\bibitem{Mur03}
 J.D. Murray, Mathematical Biology II. Spatial models and biomedical applications,
 3 ed., Interdisciplinary Applied Mathematics, 18. Springer-Verlag, New York, 2003.

\bibitem{powathil2007}
G. Powathil, M. Kohandel, S. Sivaloganathan, A. Oza and M.
Milosevic,
 Mathematical modeling of brain tumors: Effects of radiotherapy and
 chemotherapy,
 Physics in Medicine and Biology, \textbf{52} (2007) 3291--3306.


\bibitem{rockne2008modeling}
R. Rockne, E.C.Jr. Alvord, J.K. Rockhill  and  K.R. Swanson, A
mathematical model for brain tumor response to radiation therapy,
 J. Math. Biology, Special Issue on Computational Oncology, \textbf{58} (2009) 561--578.

\bibitem{rockne2010predicting}
R. Rockne, J.K. Rockhill,  M. Mrugala, A.M. Spence, I. Kalet, K.
Hendrickson, A. Lai, T. Cloughsey, E.C.Jr. Alvord and K.R. Swanson,
 Predicting the efficacy of radiotherapy in individual glioblastoma patients in vivo: a mathematical modeling
 approach,
Phys. Med.  Biol., \textbf{55}(12) (2010) 3271--3285.



\bibitem{Sim}
J. Simon. Compact sets in the space $L^p(0,T; B)$,
 Ann. Mat. Pura Appl., \textbf{146}(4) (1987) 65--96.


\bibitem{S1} G. Stampacchia,
\'{E}quations elliptiques du second ordre \`{a} coefficients
discontinus, S\'{e}minaire Jean Leray, \textbf{3} (1963--1964)
1--77.

\bibitem{swanson2003}
K.R. Swanson, C. Bridge, J. Murray  and E.C. Alvord, Virtual and
real brain tumors: Using mathematical modeling to quantify glioma
growth and invasion,
 Journal of Neurological Sciences, \textbf{216} (2003) 1--10.

 \bibitem{WXZZ21}
F. Wang, L. Xue, K. Zhao, X. Zheng, Global stabilization and
boundary control of generalized Fisher/KPP equation and application
to diffusive SIS model,
 J. Differential Equations, \textbf{275} (2021)  391--417.

\bibitem{kao2}
M. Yousefnezhad, C.-Y. Kao, S. A. Mohammadi,
 Optimal chemotherapy for brain tumor growth in a reaction-diffusion
 model,
SIAM Journal on Applied Mathematics, \textbf{81}(3) (2021)
1077--1097.

\bibitem{lima2017selection}
 E. A. Lima, J. T. Oden,  B. Wohlmuth,  A. Shahmoradi, D. A. Hormuth II,  T. E. Yankeelov, L. Scarabosio,  T. Horger. Selection and validation of predictive models of radiation effects on tumor growth based on noninvasive imaging data. Computer methods in applied mechanics and engineering. \textbf{2017} Dec 1;327:277-305.

\bibitem{Jarrett2018mathematical}
A.M. Jarrett, E. A. Lima, D. A. Hormuth, M. T. McKenna, X. Feng, D. A. Ekrut, A. C. Resende, A. Brock, T. E. Yankeelov. Mathematical models of tumor cell proliferation: A review of the literature. Expert review of anticancer therapy. 2018 Dec 2;18(12):1271-86.

\bibitem{ji2022post}
H. Ji, K. Lafata, Y. Mowery, D. Brizel, A. L. Bertozzi, F. F. Yin, C. Wang. Post-radiotherapy PET image outcome prediction by deep learning under biological model guidance: A feasibility study of oropharyngeal cancer application. Frontiers in Oncology. \textbf{2022} May 13;12:895544.

\bibitem{Yankeelov2013clinically}
T. E. Yankeelov, N. Atuegwu, D. Hormuth, J. A. Weis, S. L. Barnes, M. I. Miga, E. C. Rericha, V. Quaranta. Clinically relevant modeling of tumor growth and treatment response. Science translational medicine. 
\textbf{2013} May 29;5(187):187ps9-.

\end{thebibliography}
\end{document}